\documentclass[11pt]{article}
\usepackage[T1]{fontenc}
\usepackage[margin=1in]{geometry}
\usepackage{amsmath,amssymb,amsthm,graphicx}
\usepackage[numbers,sort&compress]{natbib}
\usepackage{mathtools}
\usepackage{xcolor}
\usepackage{enumerate}
\usepackage{bbm}
\usepackage[hidelinks]{hyperref}
\hypersetup{
    colorlinks=true,
    linkcolor=blue
}

\theoremstyle{definition}
\newtheorem{assumption}{Assumption}[section]
\newtheorem{definition}[assumption]{Definition}
\theoremstyle{plain}
\newtheorem{theorem}[assumption]{Theorem}
\newtheorem{proposition}[assumption]{Proposition}
\newtheorem{lemma}[assumption]{Lemma}
\newtheorem{remark}[assumption]{Remark}
\numberwithin{equation}{section}

\begin{document}

\title{Continuous-Time $q$-Learning with Nonsmooth\\ Value Functions: a BSDE Approach}

\author{
Manuel Asero\thanks{Università degli Studi di Milano, Milan, Italy; manuel.asero@studenti.unimi.it.}
\and
Elena Bandini\thanks{Università di Bologna, Bologna, Italy; elena.bandini7@unibo.it.}
\and
Andrea Cosso\thanks{Università degli Studi di Milano, Milan, Italy; andrea.cosso@unimi.it.}
\and
Laura Perelli\thanks{Università degli Studi di Milano, Milan, Italy; laura.perelli@unimi.it.}
}

\date{}

\maketitle

\begin{abstract}
We study continuous-time $q$-learning for entropy-regularized diffusion control without assuming classical differentiability of the value function. Rewards and admissible policies need only be Borel measurable, which allows discontinuous rewards. The diffusion coefficient $\sigma$ is action independent and locally uniformly elliptic, and the drift is $b_0+\sigma\psi$, with $b_0$ action independent and $\psi$ bounded. In this setting the $q$-function involves the value only through its gradient. For each admissible policy, we represent its value by a backward stochastic differential equation (BSDE) and prove that this value is continuous before the horizon, although the data are only measurable. Replacing the gradient with a measurable representation of the martingale integrand, we define the $q$-function up to Lebesgue-null sets and recover the classical definition when the value function is smooth. We then prove policy improvement, when the Gibbs update is admissible, and martingale characterizations of the value and the $q$-function from on-policy and off-policy observations. Here the candidate value need only be continuous before the horizon, and actions are sampled on a deterministic finite grid and held constant between sampling times. Under an integrability condition on the reward, we show that a second BSDE gives an upper bound for the optimal value, attained when its Gibbs policy is admissible, and we characterize its solution in the same way.
\end{abstract}

\noindent\textbf{Keywords.} Continuous-time reinforcement learning; $q$-learning; entropy regularization; backward stochastic differential equations.

\medskip

\noindent\textbf{MSC 2020.} 93E35, 68T05, 60H10, 93E20.

\section{Introduction}
\label{section:introduction}

Continuous-time reinforcement learning seeks to improve decisions from observations of an evolving stochastic system, without discretizing time beforehand~\cite{Doya2000}. For diffusion models, entropy regularization, a device familiar from discrete-time reinforcement learning~\cite{ZiebartEtAl2008,HaarnojaEtAl2018}, allows exploration to be incorporated into stochastic control by replacing a single action with a probability distribution over actions. Wang, Zariphopoulou, and Zhou~\cite{WangZariphopoulouZhou2020} developed this formulation and its associated exploratory dynamics, and Tang, Zhang, and Zhou~\cite{TangZhangZhou2022} analyzed the resulting Hamilton--Jacobi--Bellman equation. Jia and Zhou~\cite{JiaZhou2022} subsequently characterized policy values through martingale conditions, and derived policy gradients within the same framework~\cite{JiaZhou2022b}.

\medskip\noindent A central object of discrete-time reinforcement learning is the $Q$-function, the expected return obtained by taking a given action and following a given policy thereafter, which underlies $Q$-learning and its descendants~\cite{WatkinsDayan1992,SuttonBarto2018}. It does not survive the passage to continuous time: an action held over a vanishing time interval has a vanishing effect on the return, so that the $Q$-function collapses to the value function and no longer discriminates among actions~\cite{Baird1994,TallecBlierOllivier2019}. The continuous-time $q$-function of Jia and Zhou~\cite{JiaZhou2023} instead describes the first-order effect of an action over a short time interval and provides a basis for policy improvement and learning from observed trajectories: the Gibbs policy that it generates improves on the given one, and martingale conditions characterize it, jointly with the value function, from data collected either under the policy itself or under a different one. This framework has since been extended to mean-field control~\cite{WeiYu2025}, jump-diffusions~\cite{GaoLiZhou2026,BoHuangYuZhang2026}, and risk-sensitive criteria~\cite{Jia2026}, and the convergence of the associated algorithms has been quantified in~\cite{TangZhou2024}.

\medskip\noindent The classical definition of the $q$-function involves the time derivative, spatial gradient, and spatial Hessian of the value function. For a policy $\pi$ with value $J^\pi$, it reads
\[
    q^\pi(t,x,a)=\bigl(\partial_t+\mathcal L^a-\beta\bigr)J^\pi(t,x)+r(t,x,a),
\]
where $r$ is the running reward, $\beta$ is the discount factor, and $\mathcal L^a$ is the generator of the state under the constant action $a$ (see Remark~\ref{remark:classical-q}). Continuity of the rewards alone does not ensure the classical regularity required by this definition. Such regularity is, moreover, hard to guarantee a priori, since it is a joint property of the unknown dynamics, of the rewards, and of the policy under evaluation. Where it is available, propagating it along successive policy updates requires delicate estimates on the associated partial differential equations~\cite{HuangWangZhou2025,TranWangZhang2025}. The requirement also bears on the functions used for learning: the joint martingale characterizations in Theorems~7 and~9 of the Jia--Zhou erratum~\cite{JiaZhou2025Erratum} assume that the candidate value belongs to $C^{1,2}$. In this paper we develop the theory of the $q$-function (its definition, policy improvement, and martingale characterizations) without this requirement, and without assuming that the rewards are continuous.

\medskip\noindent We work on a finite horizon with Borel measurable rewards of polynomial growth. The reference drift and diffusion coefficient are globally Lipschitz in the state, while the controlled drift is locally Lipschitz in the state, uniformly in time and action. The diffusion coefficient is action independent and locally uniformly elliptic. With $b_0$ the reference drift and $\sigma$ the diffusion coefficient, the controlled drift is $b=b_0+\sigma\psi$, where $\psi$ is bounded (see Assumption~\ref{assumption:coefficients}). This is the structure of the weak formulation of stochastic control, in which actions affect the state through a change of probability measure~\cite{Benes1971,DavisVaraiya1973} and the value is described by a backward stochastic differential equation (BSDE)~\cite{ElKarouiPengQuenez1997,HamadeneLepeltier1995}. Policies have strictly positive densities and satisfy moment and entropy conditions (Definition~\ref{definition:policy}), but they need not be regular in time or state. For this reason the exploratory dynamics are understood in the weak sense. We thus remove the regularity assumption on the value function at the price of restrictions on the controlled dynamics. One of them, the action independence of the diffusion coefficient, is shared by much of the convergence theory of policy iteration for controlled diffusions~\cite{KerimkulovSiskaSzpruch2020,HuangWangZhou2025,MaWangZhang2026,MaWangZhangZhou2026}. Within these restrictions, discontinuous rewards are allowed, such as rewards accruing while the state lies in a target region, and so are discontinuous policies.

\medskip\noindent Smoothness can be dispensed with in this setting because the $q$-function is then a first-order object. Indeed, let $\pi$ be a policy whose value $J^\pi$ is of class $C^{1,2}$. Subtracting the linear parabolic equation satisfied by $J^\pi$ from the classical definition of its $q$-function cancels the time derivative and, because the diffusion coefficient is action independent, the Hessian as well. What remains is, for almost every $(t,x)$,
\begin{equation}\label{equation:intro-q}
q^\pi(t,x,a)=\big\langle\sigma(t,x)^\top\partial_x J^\pi(t,x),\,\psi(t,x,a)-\bar\psi_\pi(t,x)\big\rangle+r(t,x,a)-R_\pi(t,x),
\end{equation}
where $\gamma>0$ is the entropy-regularization parameter, and $\bar\psi_\pi$ and $R_\pi$ denote the averages of $\psi$ and of $r-\gamma\log\pi$ with respect to $\pi(\cdot\,|\,t,x)$. The value function thus enters only through $\sigma^\top\partial_xJ^\pi$. For a smooth value, this quantity is the martingale integrand of the BSDE that represents the value~\cite{PardouxPeng1992,ElKarouiPengQuenez1997}. The integrand, unlike the gradient, is available whether or not the value is differentiable.

\medskip\noindent The argument uses the theory of BSDEs developed by Pardoux and Peng~\cite{PardouxPeng1990}. By Girsanov's theorem, each sampled state trajectory solves, under an equivalent probability measure, the reference equation $dX_s=b_0(s,X_s)\,ds+\sigma(s,X_s)\,dB_s$, whatever the policy that generated it (Section~\ref{section:reference-dynamics}). For a fixed policy $\pi$, we represent $J^\pi$ by a BSDE along the reference diffusion $X$ (Proposition~\ref{proposition:bsde-value}). Its martingale integrand is of the form $z_\pi(s,X_s)$ for a Borel function $z_\pi$ (Proposition~\ref{proposition:markovian-integrand}), and we define $q^\pi$ by~\eqref{equation:intro-q} with $z_\pi$ in place of $\sigma^\top\partial_xJ^\pi$. Since the reference diffusion has positive transition densities (Proposition~\ref{proposition:densities}), the function $z_\pi$ is unique up to null sets, and the choice of the representative does not matter along sampled trajectories. If $J^\pi$ is of class $C^{1,2}$, we recover the definition of~\cite{JiaZhou2023} (Remark~\ref{remark:classical-q}). Although the rewards and policies are only measurable, policy values are continuous before the horizon (Proposition~\ref{proposition:continuity}). This strong Feller property, which we derive from Krylov's estimate~\cite{Krylov1980}, allows identities holding almost everywhere to be extended to every time and state.

\medskip\noindent For a fixed policy $\pi$, we prove the following results. Theorem~\ref{theorem:policy-improvement} shows, by the comparison theorem for BSDEs, that the Gibbs update of $\pi$, namely the Gibbs policy generated by $q^\pi$, improves on $\pi$ whenever it is admissible, and that an admissible fixed point of this update is optimal within the policy class. Since admissibility only requires measurability and the moment and entropy bounds of Definition~\ref{definition:policy}, we do not have to check any regularity of the update in time or state. This is essential, because the update depends on the state through the reward and through the martingale integrand, which is available only as a Borel function defined up to null sets. Theorem~\ref{theorem:q-martingale-characterization} characterizes $q^\pi$ by a martingale condition when $J^\pi$ is known. Theorem~\ref{theorem:joint-martingale-characterization} characterizes the pair $(J^\pi,q^\pi)$ by the same condition, together with the terminal condition and a normalization depending on $\pi$. These conditions involve the candidate functions and observed states, actions, and rewards. The candidate value is only required to be continuous before the horizon and of polynomial growth, and the candidate $q$-function to be Borel measurable and integrable along sampled trajectories, so that nonsmooth parametric families are not excluded. In particular, the value $J^\pi$ itself is always an admissible candidate, which need not be true when candidates are required to be of class $C^{1,2}$ (Remark~\ref{remark:example}). Where the classical argument applies It\^o's formula to the candidate value, we recover its semimartingale decomposition from the martingale property itself, by martingale representation, and conclude by uniqueness for the BSDE. Under the stated integrability conditions, a single admissible behavior policy suffices for identification, including when it differs from the policy being evaluated.

\medskip\noindent For optimization, we additionally assume that $r(t,x,a)\leq\widetilde r(t,x)+\chi(a)$, where $\widetilde r$ has polynomial growth and $\exp(\chi/\gamma)$ is integrable over the action space $A$ (Assumption~\ref{assumption:separable-reward}). This holds automatically when $A$ is bounded. The driver
\[
    f_*(t,x,y,z)=\gamma\log\left(\int_A\exp\left(\frac{\langle z,\psi(t,x,a)\rangle+r(t,x,a)}{\gamma}\right)\,da\right)-\beta y
\]
is then Lipschitz in $(y,z)$ and dominates the driver of the BSDE associated with every admissible policy (Proposition~\ref{proposition:optimal-driver}). Hence the initial value $u_*(t,x)$ of the BSDE with driver $f_*$ is an upper bound for $J^\pi(t,x)$ for every admissible policy $\pi$, and the bound is attained when the associated Gibbs density has an admissible representative (Proposition~\ref{proposition:optimal-verification}). Theorem~\ref{theorem:optimal-martingale-characterization} characterizes $u_*$ and the associated function $q_*$ by a martingale condition, the terminal condition, and the normalization $\int_A\exp(q/\gamma)\,da=1$, without reference to any prescribed policy.

\medskip\noindent Throughout the paper, actions are sampled at finitely many deterministic times and held constant between successive draws, as in the Jia--Zhou erratum~\cite{JiaZhou2025Erratum}. This avoids the measurability problem raised by drawing an independent action at every instant~\cite{SzpruchTreetanthiploetZhang2024,BenderThuan2024}. The martingale properties in Theorems~\ref{theorem:q-martingale-characterization}, \ref{theorem:joint-martingale-characterization}, and~\ref{theorem:optimal-martingale-characterization} hold exactly on every sampling grid, and we do not let the mesh tend to zero. The policy value $J^\pi$ is instead defined through the exploratory dynamics, whose relation to discretely sampled policies in the vanishing-mesh limit is studied by Jia, Ouyang, and Zhang~\cite{JiaOuyangZhang2025} for sufficiently regular coefficients and policies; see also~\cite{BenderThuan2026,CarmonaLauriere2025}. On a fixed sampling grid, the expected entropy-regularized reward along a trajectory sampled from $\pi$ need not coincide with $J^\pi$.

\medskip\noindent In Remark~\ref{remark:example} we consider a one-dimensional example with running reward $\mathbbm 1_{[-\delta,\delta]}(x)-a^2/2$. The value and the $q$-function of the uniform policy are explicit. The value has a bounded Lipschitz gradient, but its second spatial derivative jumps at $x=\pm\delta$ for every $t<T$, so that the value is not of class $C^{1,2}$. The $q$-function is nevertheless continuous. In this example the classical definition is not available, while the results of the present paper apply.

\medskip\noindent We now compare our results with the literature on policy iteration for controlled diffusions. Kerimkulov, \v{S}i\v{s}ka, and Szpruch~\cite{KerimkulovSiskaSzpruch2020} introduced BSDE techniques into the analysis of Howard's algorithm, and Possama\"i and Tangpi~\cite{PossamaiTangpi2024} used them to design and analyze a policy iteration scheme for non-Markovian problems in weak formulation. For entropy-regularized problems, the convergence of iterated Gibbs updates was proved by Huang, Wang, and Zhou~\cite{HuangWangZhou2025} and by Tran, Wang, and Zhang~\cite{TranWangZhang2025} through H\"older and Sobolev estimates for the iterated equations, and by Ma, Wang, and Zhang~\cite{MaWangZhang2026} through probabilistic representations of the derivatives of the value. Closest to our setting is the recent work of Ma, Wang, Zhang, and Zhou~\cite{MaWangZhangZhou2026} on unknown drifts and running rewards with a known, action-independent diffusion coefficient. These authors learn the gradient of the optimal value function together with a function that, for such a diffusion coefficient, is essentially the optimal $q$-function. They rely on probabilistic representations along an auxiliary diffusion that involves the diffusion coefficient only, and prove convergence of the resulting algorithms under the assumption that the exploratory Hamilton--Jacobi--Bellman equation has a classical solution with bounded derivatives. All four works on entropy-regularized problems rely on regularity of the coefficients, of the rewards, or of the value function, and address the convergence of iterative schemes. We do not study convergence. We show that the $q$-function can be defined, and that Theorems~\ref{theorem:policy-improvement}, \ref{theorem:q-martingale-characterization}, and~\ref{theorem:joint-martingale-characterization} hold, without regularity assumptions on rewards, policies, or values. Finally, Mou~\cite{Mou2026} studies off-policy learning of value and advantage functions for diffusions observed at discrete times. That paper proposes a fitted $q$-learning algorithm with function approximation and proves oracle inequalities for its estimation error. The analysis relies on the ellipticity of the diffusion, which is also assumed here.

\medskip\noindent The rest of the paper is organized as follows. Section~\ref{section:framework} gives the framework and notation. Section~\ref{section:bsde-and-q} constructs the $q$-function and proves the continuity of policy values, while Section~\ref{section:policy-improvement} proves policy improvement. Sections~\ref{section:policy-evaluation} and~\ref{section:optimal-value-and-q} treat martingale characterizations and optimization, respectively.

\section{Model formulation}
\label{section:framework}

\subsection{Basic notation and stochastic framework}

Throughout the paper, we fix a finite time horizon $T>0$ and positive integers $d$, $n$, and $k$, with $n\geq d$. The state space is $\mathbb{R}^d$, while $A\subseteq\mathbb{R}^k$ denotes a Borel action space with positive Lebesgue measure. Unless otherwise specified, all integrals over $A$ are taken with respect to Lebesgue measure $da$. For vectors $x,y$ in a Euclidean space, we denote by $\langle x,y\rangle$ and $|x|$ the Euclidean inner product and norm, respectively. For a matrix $M$, we write $M^\top$ for its transpose and $|M|^2=\operatorname{tr}(MM^\top)$ for its squared Frobenius norm.

\medskip\noindent Let $(\Omega,\mathcal{F},\mathbb{P})$ be a complete probability space supporting an $n$-dimensional Brownian motion $W=(W_s)_{0\leq s\leq T}$ and a sequence $(U_i)_{i\geq0}$ of independent random variables, uniformly distributed on $[0,1]$ and independent of $W$. We denote by $\mathbb{F}^W=(\mathcal{F}^W_s)_{0\leq s\leq T}$ the augmented natural filtration of $W$. We write $\mathbb{E}$ for expectation under $\mathbb{P}$ and, more generally, $\mathbb{E}^{\mathbb Q}$ for expectation under any probability measure $\mathbb Q$ on $(\Omega,\mathcal F)$. Given $0\leq t\leq T$, a probability measure $\mathbb Q$, a filtration $\mathbb{H}$, $p\geq1$, an integer $N\geq1$ and a subset $E\subseteq \mathbb R^N$, we denote by $\mathcal{S}^p_{t,T}(E;\mathbb Q,\mathbb{H})$ the space of continuous, $\mathbb{H}$-progressively measurable $E$-valued processes $Y$ such that
\[ 
    \mathbb{E}^{\mathbb Q} \bigg[\sup_{t\leq s\leq T}|Y_s|^p\bigg]<+\infty, 
\] 
by $\mathcal M^p_{t,T}(E;\mathbb Q,\mathbb{H})$ the space of $\mathbb{H}$-progressively measurable $E$-valued processes $Z$ such that
\[ 
    \mathbb{E}^{\mathbb Q} \bigg[\int_t^T |Z_s|^p\,ds\bigg]<+\infty,
\]
and by $\mathcal H^p_{t,T}(E;\mathbb Q,\mathbb{H})$ the space of $\mathbb{H}$-progressively measurable $E$-valued processes $Z$ such that
\[ 
    \mathbb{E}^{\mathbb Q} \bigg[\bigg(\int_t^T |Z_s|^2\,ds\bigg)^{p/2}\bigg]<+\infty.
\]
Processes in $\mathcal{S}^p$ are identified up to indistinguishability, whereas processes in $\mathcal M^p$ and $\mathcal H^p$ are identified up to $ds\otimes d\mathbb Q$-almost everywhere equality. For deterministic functions, ``almost everywhere'' refers to Lebesgue measure on the underlying Euclidean domain.

\subsection{Controlled diffusion and rewards} 
For an initial condition $(t,x)\in[0,T]\times\mathbb{R}^d$ and an $A$-valued control process $\alpha=(\alpha_s)_{t\leq s\leq T}$, progressively measurable with respect to a filtration for which $W$ is a Brownian motion, we consider the controlled diffusion 
\begin{equation}\label{equation:controlled-diff} 
\begin{cases}
    dX_s^{t,x,\alpha} = b(s,X_s^{t,x,\alpha},\alpha_s)\,ds + \sigma(s,X_s^{t,x,\alpha})\,dW_s,&s\in[t,T],\\
    X_t^{t,x,\alpha}=x.
\end{cases}
\end{equation} 
The diffusion coefficient is independent of the action, while the drift is of the form
\begin{equation}\label{equation:drift-structure} 
    b(t,x,a) = b_0(t,x) + \sigma(t,x)\psi(t,x,a) 
\end{equation} 
for suitable functions $b_0:[0,T]\times\mathbb R^d\to\mathbb R^d$,
$\sigma:[0,T]\times\mathbb R^d\to\mathbb R^{d\times n}$, and
$\psi:[0,T]\times\mathbb R^d\times A\to\mathbb R^n$.

\medskip\noindent Let $r: [0,T]\times\mathbb{R}^d\times A\to\mathbb{R}$ and $g:\mathbb{R}^d\to\mathbb{R}$ denote, respectively, the running and terminal reward functions Given a discount factor $\beta\geq0$, the discounted reward associated with $\alpha$ is defined, whenever \eqref{equation:controlled-diff} is well posed and the expression below is integrable, by
\begin{equation*}
    \mathcal{J}(t,x,\alpha) \coloneqq \mathbb{E}\left[ \int_t^T e^{-\beta(s-t)} r(s,X_s^{t,x,\alpha},\alpha_s)\,ds + e^{-\beta(T-t)} g(X_T^{t,x,\alpha}) \right].
\end{equation*}
The functional $\mathcal{J}$ is introduced to specify the underlying state--reward model. Our main object will be the entropy-regularized value associated with a randomized feedback policy, introduced below.

\medskip\noindent Throughout the paper, we consider the following standing assumptions.

\begin{assumption}\label{assumption:coefficients}
The functions $b_0$, $\sigma$, $\psi$, $r$, and $g$ are Borel measurable. The functions $\psi$, $b_0(\cdot,0)$,
and $\sigma(\cdot,0)$ are bounded. There exist constants $L,C>0$ and
$m\geq1$, such that
\begin{align}
&|b_0(t,x)-b_0(t,y)|+|\sigma(t,x)-\sigma(t,y)|
 \leq L|x-y|,
\label{equation:coefficient-lipschitz}\\
&|r(t,x,a)|\leq C(1+|x|^m+|a|^m),
\qquad |g(x)|\leq C(1+|x|^m),
\label{equation:reward-growth}
\end{align}
for every $t\in[0,T]$, $a\in A$, 
$x,y\in\mathbb R^d$. Furthermore, for every $R>0$ there exist constants $L_R,\lambda_R>0$ such that
\begin{align}
&|b(t,x,a)-b(t,y,a)|
 \leq L_R|x-y|,\label{equation:coefficient-lipschitz2}\\
&\big\langle \sigma(t,x)\sigma(t,x)^\top v,v \big\rangle
  \geq \lambda_R |v|^2,
  \label{equation:ellipticity}
\end{align}
for every $t\in[0,T]$, $a\in A$, $|x|,|y|\leq R$, $v\in\mathbb R^d$.
\end{assumption}

\noindent The global Lipschitz continuity of $b_0$ and $\sigma$, together with the boundedness of $b_0(\cdot,0)$ and $\sigma(\cdot,0)$, implies that both coefficients have at most linear growth. Since $\psi$ is bounded, \eqref{equation:drift-structure} further yields 
\begin{equation}\label{equation:coefficient-growth} 
    |b_0(t,x)| + |b(t,x,a)| + |\sigma(t,x)| \leq C(1+|x|),
\end{equation} 
uniformly in $(t,a)$, possibly after changing the value of $C$. These conditions place the forward state equation in the classical Markovian diffusion setting used in the BSDE approach of \cite[Section~4]{ElKarouiPengQuenez1997}. In particular, for every control of this type, including the sampled controls considered below, the corresponding state equation admits a unique strong solution with finite moments of every order. Local uniform ellipticity in
\eqref{equation:ellipticity} implies that $\sigma$ has rank $d$, which explains the restriction $n\geq d$. This condition will be
used later in the identification of the $q$-function and in the proof that policy values are continuous.

\begin{remark}\label{remark:indicator-rewards}
Assumption~\ref{assumption:coefficients} requires no regularity of the rewards $r$ and $g$ beyond Borel measurability, whereas both are continuous in \cite[Assumption~1]{JiaZhou2023}. A relevant class of running rewards covered here, but not there, is
\[
    r(t,x,a)=c\,\mathbbm 1_D(x)+r_0(t,x,a),
\]
where $c\neq0$, $D\subset\mathbb R^d$ is a Borel set, and $r_0$ is continuous and satisfies \eqref{equation:reward-growth}. The first term rewards or penalizes (depending on the sign of $c$) the time spent by the state in the target region $D$. For such rewards the value of a policy is, in general, not of class $C^{1,2}$. See Remark~\ref{remark:example} for an example. 
\end{remark}

\subsection{Randomized feedback policies} We next introduce the exploratory randomization of the action. 
\begin{definition} \label{definition:policy} 
An admissible policy is a Borel function  $(t,x,a) \longmapsto \pi(a\,|\, t,x)\in(0,+\infty)$ such that 
\begin{equation*}
    \int_A \pi(a\,|\, t,x)\,da=1, \qquad (t,x)\in[0,T]\times\mathbb{R}^d, 
\end{equation*} 
and satisfying the following conditions. For every $p\geq1$, there exist constants $C_{\pi,p}>0$ and $m_{\pi,p}\geq1$ such that \begin{equation} \label{equation:policy-moments} 
        \int_A |a|^p\pi(a\,|\, t,x)\,da \leq C_{\pi,p} \bigl(1+|x|^{m_{\pi,p}}\bigr),\qquad\text{for all $(t,x)\in[0,T]\times\mathbb{R}^d$. }
    \end{equation} 
    Moreover, there exist $C_\pi>0$ and $m_\pi\geq1$ such that 
    \begin{equation}\label{equation:entropy-integrability} 
        \int_A |\log\pi(a\,|\, t,x)| \pi(a\,|\, t,x)\,da \leq C_\pi(1+|x|^{m_\pi}).
    \end{equation}
\end{definition} 
\noindent We denote by $\Pi$ the collection of admissible policies. Notice that $\Pi$ is nonempty: the density on $A$ proportional to $\exp(-|a|^2)$, independent of $(t,x)$, belongs to $\Pi$. 

\begin{remark}
The bounds \eqref{equation:policy-moments} and \eqref{equation:entropy-integrability} are the moment and entropy conditions of the framework of \cite{JiaZhou2023}. In contrast with \cite[Definition~1]{JiaZhou2023}, no continuity of the map $(t,x)\longmapsto\pi(\cdot\,|\,t,x)$ is required: neither continuity in $L^1(A,da)$, nor Lipschitz continuity in the state with respect to the total variation distance. The analysis developed below only uses that the averaged coefficients introduced in the next subsection are Borel measurable, together with their growth. For the exploratory dynamics introduced below, we only need, and only prove, existence and uniqueness in law; strong solutions are not used. They are available, by the classical Lipschitz theory, when $\pi$ is in addition Lipschitz continuous in the state with respect to the total variation distance, uniformly in time, since the averaged drift is then locally Lipschitz in the state.
\end{remark}

\subsection{Exploratory dynamics and policy evaluation} Following \cite{JiaZhou2023}, a randomized policy is incorporated into the state dynamics by averaging the local characteristics with respect to its action distribution. In our setting only the drift has to be averaged, since the diffusion coefficient does not depend on the action. For $\pi\in\Pi$, define 
\begin{gather*} 
    \bar{\psi}_\pi(t,x) \coloneqq \int_A \psi(t,x,a)\pi(a\,|\, t,x)\,da,\qquad 
    \bar r_\pi(t,x) \coloneqq \int_A r(t,x,a)\pi(a\,|\, t,x)\,da,\\
    \ell_\pi(t,x) \coloneqq \int_A \log\pi(a\,|\, t,x) \pi(a\,|\, t,x)\,da.
\end{gather*} 
Let $\gamma>0$ denote the entropy-regularization parameter. We further set 
\begin{equation*}
    R_\pi(t,x) \coloneqq \bar r_\pi(t,x) - \gamma\ell_\pi(t,x). 
\end{equation*} 
The quantity $-\ell_\pi(t,x)$ is the differential entropy of $\pi(\cdot\,|\,t,x)$. Notice that, under Assumption~\ref{assumption:coefficients} and by the properties of $\pi$, the functions $\bar\psi_\pi$, $\bar r_\pi$, and $\ell_\pi$ are Borel measurable, $\bar\psi_\pi$ is bounded, and $R_\pi$ has polynomial growth.

\noindent The exploratory state process associated with $\pi$ is a solution of \begin{equation} 
\begin{cases}\label{equation:aggregate-state} 
    d\bar X_s^{t,x,\pi} = \Big[ b_0(s,\bar X_s^{t,x,\pi}) + \sigma(s,\bar X_s^{t,x,\pi}) \bar\psi_\pi(s,\bar X_s^{t,x,\pi}) \Big]\,ds+ \sigma(s,\bar X_s^{t,x,\pi})\,dW_s,&s\in[t,T],\\ \bar X_t^{t,x,\pi}=x. 
\end{cases} 
\end{equation} 
Equivalently, its drift is 
\[ 
    \bar b_\pi(t,x)\coloneqq\int_A b(t,x,a) \pi(a\,|\, t,x)\,da = b_0(t,x)+\sigma(t,x)\bar\psi_\pi(t,x).
\] 
Since $\bar\psi_\pi$ is bounded but, in general, only Borel measurable, equation \eqref{equation:aggregate-state} is understood in the weak sense: a solution consists of a filtered probability space supporting an $n$-dimensional Brownian motion, still denoted by $W$, and a continuous adapted process $\bar X^{t,x,\pi}$ satisfying \eqref{equation:aggregate-state}. By Proposition~\ref{proposition:aggregate-value} below, such a solution exists and is unique in law.
The entropy-regularized value associated with $\pi$ is 
\begin{equation}\label{equation:policy-value} 
    J^\pi(t,x) \coloneqq \mathbb{E}\left[ \int_t^T e^{-\beta(s-t)} R_\pi(s,\bar X_s^{t,x,\pi})\,ds + e^{-\beta(T-t)} g(\bar X_T^{t,x,\pi}) \right].
\end{equation} 
Here $\mathbb E$ denotes the expectation on the probability space of an arbitrary weak solution of \eqref{equation:aggregate-state}; by uniqueness in law, $J^\pi(t,x)$ does not depend on this choice.
This is the policy value used throughout the paper.

\begin{proposition} \label{proposition:aggregate-value} Let $\pi\in\Pi$. For every $(t,x)\in[0,T]\times\mathbb R^d$, equation \eqref{equation:aggregate-state} admits a weak solution, which is unique in law. Moreover, for every $p\geq2$, \begin{equation} \label{equation:aggregate-state-moments} 
    \mathbb{E}\bigg[ \sup_{t\leq s\leq T} |\bar X_s^{t,x,\pi}|^p \bigg] \leq C_p(1+|x|^p),
\end{equation} 
where $C_p$ is independent of $\pi$, $t$, and $x$. Consequently, $J^\pi$ is well defined, finite, satisfies $J^\pi(T,x)=g(x)$, and has polynomial growth: there exist constants $C_\pi>0$ and $m_\pi\geq1$ such that 
\[ 
    |J^\pi(t,x)| \leq C_\pi(1+|x|^{m_\pi}). 
\] 
\end{proposition}
\begin{proof}
Let $\mathbb Q$ be a probability measure under which $B$ is an $n$-dimensional Brownian motion, and let $X$ be the unique strong solution of
\[
    dX_s=b_0(s,X_s)\,ds+\sigma(s,X_s)\,dB_s,
    \qquad s\in[t,T],
    \qquad X_t=x,
\]
which exists by \eqref{equation:coefficient-lipschitz} and \eqref{equation:coefficient-growth}. Since $\bar\psi_\pi$ is bounded, Novikov's condition is satisfied, and
\[
    \frac{d\widetilde{\mathbb Q}^{\,\pi}}{d\mathbb Q}
    \coloneqq
    \exp\left(
        \int_t^T\langle\bar\psi_\pi(v,X_v),dB_v\rangle
        -\frac12\int_t^T|\bar\psi_\pi(v,X_v)|^2\,dv
    \right)
\]
defines a probability measure equivalent to $\mathbb Q$. By Girsanov's theorem, the process $\widetilde B_s\coloneqq B_s-\int_t^s\bar\psi_\pi(v,X_v)\,dv$ is a Brownian motion under $\widetilde{\mathbb Q}^{\,\pi}$, and
\[
    dX_s
    =
    \Big[
        b_0(s,X_s)+\sigma(s,X_s)\bar\psi_\pi(s,X_s)
    \Big]\,ds
    +\sigma(s,X_s)\,d\widetilde B_s.
\]
Hence $(X,\widetilde B)$ is a weak solution of \eqref{equation:aggregate-state} under $\widetilde{\mathbb Q}^{\,\pi}$.

\noindent To prove uniqueness in law, let $(\bar X,W)$ be any weak solution of \eqref{equation:aggregate-state}, defined on a probability space with probability measure $\bar{\mathbb P}$. By Girsanov's theorem, the process $\widehat W_s\coloneqq W_s+\int_t^s\bar\psi_\pi(v,\bar X_v)\,dv$, $s\in[t,T]$, is a Brownian motion under the probability measure $\widehat{\mathbb P}$ defined by
\[
    \frac{d\widehat{\mathbb P}}{d\bar{\mathbb P}}
    \coloneqq
    \exp\left(
        -\int_t^T\langle\bar\psi_\pi(v,\bar X_v),dW_v\rangle
        -\frac12\int_t^T|\bar\psi_\pi(v,\bar X_v)|^2\,dv
    \right),
\]
and $\bar X$ solves, under $\widehat{\mathbb P}$, the equation satisfied by $X$, driven by $\widehat W$. Pathwise uniqueness holds for that equation. By the Yamada--Watanabe theorem \cite[Proposition~5.3.20 and Corollary~5.3.23]{KaratzasShreve1991}, the pair $(\bar X,\widehat W)$ has therefore the same law under $\widehat{\mathbb P}$ as the pair $(X,B)$ under $\mathbb Q$. Moreover,
\[
    \frac{d\bar{\mathbb P}}{d\widehat{\mathbb P}}
    =
    \exp\left(
        \int_t^T\langle\bar\psi_\pi(v,\bar X_v),d\widehat W_v\rangle
        -\frac12\int_t^T|\bar\psi_\pi(v,\bar X_v)|^2\,dv
    \right)
\]
is the same functional of $(\bar X,\widehat W)$ as $d\widetilde{\mathbb Q}^{\,\pi}/d\mathbb Q$ is of $(X,B)$. Consequently, the law of $\bar X$ under $\bar{\mathbb P}$ coincides with the law of $X$ under $\widetilde{\mathbb Q}^{\,\pi}$.

\noindent Since $\psi$ is bounded, the drift $\bar b_\pi$ satisfies the linear growth bound \eqref{equation:coefficient-growth} with a constant independent of $\pi$. The Burkholder--Davis--Gundy inequality and Gronwall's lemma, therefore give \eqref{equation:aggregate-state-moments}. Together with the polynomial growth of $R_\pi$ and $g$, this shows that both terms in \eqref{equation:policy-value} are integrable and that, for suitable constants $C_\pi>0$ and $m_\pi\geq1$,
\[
    |J^\pi(t,x)|
    \leq
    C_\pi(1+|x|^{m_\pi}).
\]
Finally, $J^\pi(T,x)=g(x)$ follows directly from \eqref{equation:policy-value}.
\end{proof}

\subsection{Discretely sampled realization of a policy}
Following \cite{JiaZhou2025Erratum}, we generate the state--action trajectories used as observations by sampling actions from a feedback policy at finitely many times and keeping each sampled action fixed until the next sampling time. This avoids the measurability difficulty associated with continuum independent sampling of actions. Fix $(t,x)\in[0,T)\times\mathbb{R}^d$, a behavior policy $\mu\in\Pi$, and a deterministic grid 
\[ 
    G = \{t=t_0<t_1<\cdots<t_{N_G}=T\}, \qquad |G| \coloneqq \max_{0\leq i<N_G}(t_{i+1}-t_i). 
\] We denote by $\mathbb F^G=(\mathcal F_s^G)_{0\leq s\leq T}$ the usual augmentation of the filtration generated at time $s$ by 
\[ 
    \sigma\left( W_v:0\leq v\leq s; \, U_i:0\leq i<N_G,\ t_i\leq s \right). 
\]
Then, $U_i$ becomes available at time $t_i$. Since $(U_i)_{i\geq0}$ is independent of $W$, the latter remains a Brownian motion with respect to $\mathbb F^G$. By \cite[Lemma~4.22]{Kallenberg2021}, there exists a Borel map $F_\mu: [0,T]\times\mathbb R^d\times[0,1] \longrightarrow A$ such that, for every $(s,y)\in[0,T]\times\mathbb R^d$ and every uniform random variable $U$ on $[0,1]$, $F_\mu(s,y,U) \sim \mu(\cdot\,|\,s,y)$. We recursively define the sampled action process by $\alpha_s^{\mu,G} \coloneqq F_\mu ( t_i, X_{t_i}^{t,x,\mu,G}, U_i)$, $s\in[t_i,t_{i+1})$, and let $X^{t,x,\mu,G}$ solve 
\begin{equation}\label{equation:sampled-state} 
\begin{cases}
    dX_s^{t,x,\mu,G} = b(s,X_s^{t,x,\mu,G},\alpha_s^{\mu,G})\,ds + \sigma(s,X_s^{t,x,\mu,G})\,dW_s,&s\in[t,T],\\ X_t^{t,x,\mu,G}=x,
\end{cases}
\end{equation}
successively on the intervals of the grid. We set $\alpha_T^{\mu,G} \coloneqq \alpha_{t_{N_G-1}}^{\mu,G}$. For each $i$, the random variable $X_{t_i}^{t,x,\mu,G}$ depends only on the Brownian path up to $t_i$ and on the previous randomizations $U_0,\ldots,U_{i-1}$. Hence $U_i$ is independent of $X_{t_i}^{t,x,\mu,G}$ and, conditionally on the state, 
\begin{equation}\label{equation:sampled-action-law} 
    \mathcal L_{\mathbb P} \big( \alpha_{t_i}^{\mu,G} \,\big|\, X_{t_i}^{t,x,\mu,G} \big) = \mu \big( a \,\big|\, t_i,X_{t_i}^{t,x,\mu,G} \big)da.
\end{equation} 

\begin{proposition} \label{proposition:sampled-state} 
For every $(t,x)\in[0,T)\times\mathbb R^d$, $\mu\in\Pi$, and grid $G$, the preceding construction yields an $\mathbb F^G$-progressively measurable $A$-valued process $\alpha^{\mu,G}$ and a unique strong solution $X^{t,x,\mu,G}$ of \eqref{equation:sampled-state}. Moreover, for every $p\geq2$, 
\begin{equation}\label{equation:sampled-state-moments} 
    \mathbb E\bigg[ \sup_{t\leq s\leq T} |X_s^{t,x,\mu,G}|^p \bigg] \leq C_p(1+|x|^p),
\end{equation} 
where $C_p$ is independent of $t$, $x$, $\mu$, and $G$. Finally, for every $p\geq1$, $\alpha^{\mu,G} \in \mathcal M^p_{t,T} (A;\mathbb P,\mathbb F^G)$. 
\end{proposition}
\begin{proof}
On each interval $[t_i,t_{i+1})$, the sampled action is constant
and $\mathcal F_{t_i}^G$-measurable. For this given random action,
\eqref{equation:coefficient-lipschitz} and \eqref{equation:coefficient-lipschitz2} give local Lipschitz
continuity in the state with constants independent of its value.
Together with the uniform linear-growth bound
\eqref{equation:coefficient-growth}, localization yields a unique
nonexplosive strong solution on the interval. Proceeding
recursively over the finite grid defines $X^{t,x,\mu,G}$ and the
progressively measurable process $\alpha^{\mu,G}$ on $[t,T]$.

\noindent Let
\[
 \tau_R\coloneqq\inf\{s\in[t,T]:|X_s^{t,x,\mu,G}|\geq R\}\wedge T.
\]
By the Burkholder--Davis--Gundy inequality, the linear-growth bound \eqref{equation:coefficient-growth}, and Gronwall's lemma, we obtain, for every $p\geq2$,
\[
 \mathbb E\bigg[\sup_{t\leq s\leq T}
 |X_{s\wedge\tau_R}^{t,x,\mu,G}|^p\bigg]
 \leq C_p(1+|x|^p),
\]
where $C_p$ is independent of $R$, $\mu$, and $G$. Letting
$R\to\infty$ proves \eqref{equation:sampled-state-moments}.
Finally, by \eqref{equation:sampled-action-law},
\[
 \mathbb E\left[\int_t^T|\alpha_s^{\mu,G}|^p\,ds\right]
 =\sum_{i=0}^{N_G-1}(t_{i+1}-t_i)
   \mathbb E\left[\int_A|a|^p
   \mu(a\,|\,t_i,X_{t_i}^{t,x,\mu,G})\,da\right].
\]
This is finite by \eqref{equation:policy-moments} and the state moment bounds, which proves that
$\alpha^{\mu,G}\in\mathcal M^p_{t,T}(A;\mathbb P,\mathbb F^G)$.
\end{proof}

\noindent The two state processes $\bar X^{t,x,\pi}$ and $X^{t,x,\mu,G}$ introduced above have distinct roles. The exploratory process $\bar X^{t,x,\pi}$ determines the entropy-regularized value $J^\pi$ of the target policy and will enter our BSDE formulation. The pair $(X^{t,x,\mu,G},\alpha^{\mu,G})$ represents an observable state--action trajectory generated under the behavior policy $\mu$. The data are on-policy when $\mu=\pi$ and off-policy when $\mu\neq\pi$. The relation between discretely sampled policies and the exploratory formulation in the vanishing-mesh limit $|G|\to0$ is studied in \cite{JiaOuyangZhang2025}, under additional regularity assumptions on the coefficients and on the policy.

\section{BSDE representation and the \texorpdfstring{$q$}{q}-function}
\label{section:bsde-and-q}

We represent the value of a fixed policy by a BSDE. Its martingale integrand will then
provide a definition of the $q$-function without assuming
classical differentiability of the value function.

\subsection{Reference diffusion}\label{section:reference-dynamics}

We first introduce the reference diffusion which will be used throughout the paper. Let $\mathbb{Q}$ be a probability measure under which $B$ is an $n$-dimensional Brownian motion, and let $\mathbb{F}^B$
denote its augmented natural filtration. For
$(t,x)\in[0,T]\times\mathbb{R}^d$, consider
\begin{equation}\label{equation:reference-diff}
\begin{cases}
    dX_s^{t,x}=b_0(s,X_s^{t,x})\,ds+\sigma(s,X_s^{t,x})\,dB_s,&s\in[t,T],\\
    X_t^{t,x}=x.
\end{cases}
\end{equation}
By Assumption~\ref{assumption:coefficients},
\eqref{equation:reference-diff} admits a unique strong solution, which has finite moments of every order. The law of $X^{t,x}$ does not depend on the choice of $(\mathbb Q,B)$, which we omit from the notation.

\medskip\noindent We next show that every discretely sampled state process introduced in Section~\ref{section:framework} can be represented as a solution of the reference equation \eqref{equation:reference-diff} under a
suitable equivalent probability measure. Indeed, fix $\mu\in\Pi$, $(t,x)\in[0,T)\times\mathbb{R}^d$, and a sampling grid $G$. Recall that, under $\mathbb{P}$,
$X^{t,x,\mu,G}$ satisfies
\begin{equation}\label{equation:sampled-state-girsanov}
\begin{cases}
    \begin{aligned}
    dX_s^{t,x,\mu,G}={}&\Big[b_0(s,X_s^{t,x,\mu,G})
    +\sigma(s,X_s^{t,x,\mu,G})\psi(s,X_s^{t,x,\mu,G},\alpha_s^{\mu,G})\Big]\,ds\\
    &+\sigma(s,X_s^{t,x,\mu,G})\,dW_s,\qquad s\in[t,T],
    \end{aligned}\\[1ex]
    X^{t,x,\mu,G}_t = x.
\end{cases}
\end{equation}
Define, for $s\in[t,T]$,
\[
    L_s^{\mu,G}\coloneqq\exp\left(-\int_t^s
    \left\langle
    \psi(v,X_v^{t,x,\mu,G},\alpha_v^{\mu,G}),
    dW_v\right\rangle-\frac12\int_t^s
    \left|\psi(v,X_v^{t,x,\mu,G},\alpha_v^{\mu,G})
    \right|^2\,dv\right).   
\]
Since $\psi$ is bounded, Novikov's condition is satisfied. Hence $L^{\mu,G}$ is a strictly positive martingale with $L_t^{\mu,G}=1$, and
\[
    \frac{d\mathbb{Q}^{\mu,G}}{d\mathbb{P}}
    \bigg|_{\mathcal{F}_T^G}=L_T^{\mu,G}
\]
defines a probability measure $\mathbb{Q}^{\mu,G}$ on $\mathcal{F}_T^G$ which is equivalent to $\mathbb P$. The dependence of $\mathbb{Q}^{\mu,G}$ on the initial condition $(t,x)$ is suppressed in the notation. By Girsanov's theorem, the process
\begin{equation}\label{equation:reference-brownian-motion}
    B_s^{\mu,G}\coloneqq 
\begin{cases}
    W_s,&s\in[0,t),\\
    W_s+\int_t^s
    \psi(v,X_v^{t,x,\mu,G},\alpha_v^{\mu,G})\,dv,&s\in[t,T],
\end{cases}
\end{equation}
is a Brownian motion under $\mathbb{Q}^{\mu,G}$ with respect to $\mathbb{F}^G$.

\noindent Substituting $dW_s=dB_s^{\mu,G}
-\psi(s,X_s^{t,x,\mu,G},\alpha_s^{\mu,G})\,ds$
into \eqref{equation:sampled-state-girsanov} yields
\begin{equation*}
\begin{cases}
    dX_s^{t,x,\mu,G}=b_0(s,X_s^{t,x,\mu,G})\,ds
    +\sigma(s,X_s^{t,x,\mu,G})\,dB_s^{\mu,G},&s\in[t,T],\\
    X_t^{t,x,\mu,G}=x.
\end{cases}
\end{equation*}
Thus, under $\mathbb Q^{\mu,G}$, the sampled state process solves the reference SDE \eqref{equation:reference-diff}, with driving
Brownian motion $B^{\mu,G}$. In particular, by strong existence and pathwise uniqueness for \eqref{equation:reference-diff}, $X^{t,x,\mu,G}$ coincides $\mathbb Q^{\mu,G}$-almost surely with
the unique strong solution of the reference equation driven by $B^{\mu,G}$. Consequently, $X^{t,x,\mu,G}$ is adapted to the augmented natural filtration $\mathbb F^{B^{\mu,G}}$.

\medskip\noindent The change of measure also allows us to transfer integrability between $\mathbb P$ and $\mathbb Q^{\mu,G}$. For this purpose, we will repeatedly use the following uniform moment estimates for the density and its inverse: for every $p\geq1$,
\begin{equation}\label{equation:density-moments}
    \mathbb E\big[(L_T^{\mu,G})^p\big]
    +
    \mathbb E^{\mathbb Q^{\mu,G}}
    \big[(L_T^{\mu,G})^{-p}\big]
    \leq C_p,
\end{equation}
where $C_p$ is independent of $\mu$, $G$, $t$, and $x$. These bounds follow directly from the boundedness of $\psi$.

\begin{remark}
\label{remark:reference-randomization}
The change of measure does not affect the  variables $(U_i)_{i\geq 0}$, that remain independent, uniformly distributed on $[0,1]$, and independent of $B^{\mu,G}$ under $\mathbb Q^{\mu,G}$ for every choice of $\mu, G$. More precisely, let $\mathcal U=\sigma(U_i:i\geq0)$, and extend $\mathbb Q^{\mu,G}$ to $\mathcal F$ by
$d\mathbb Q^{\mu,G}=L_T^{\mu,G}\,d\mathbb P$.
Enlarging the Brownian filtration at time zero with $\mathcal U$ leaves $W$ a Brownian motion under $\mathbb P$. Since $\psi$ is bounded, $L^{\mu,G}$ is a true martingale
in this enlarged filtration. As $L_t^{\mu,G}=1$, we obtain
\[
    \mathbb E\big[L_T^{\mu,G}\mid\mathcal U\big]=1.
\]
Consequently, $\mathbb Q^{\mu,G}$ and $\mathbb P$ coincide
on $\mathcal U$.
Moreover, Girsanov's theorem guarantees that $B^{\mu,G}$ is a Brownian motion
in the same enlarged filtration under $\mathbb Q^{\mu,G}$,
and hence independent of $\mathcal U$.
\end{remark}

\noindent We conclude this subsection with a support property of the reference and sampled state--action processes that will be used repeatedly below. Its role is to pass between identities holding along sampled trajectories and Lebesgue-almost everywhere identities in the time--state--action variables. Local uniform ellipticity of $\sigma$ ensures that the reference diffusion has strictly positive transition densities, which leads to the following result.

\begin{proposition}
\label{proposition:densities}
Fix $(t,x)\in[0,T)\times\mathbb R^d$. Let $X^{t,x}$ be the reference diffusion
\eqref{equation:reference-diff} on a Brownian probability
space $(\mathbb Q,B)$. For every $s\in(t,T)$, $X_s^{t,x}$ has a density $\rho^{t,x}(s,\cdot)$, with $(s,y)\longmapsto\rho^{t,x}(s,y)$ jointly continuous on $(t,T)\times\mathbb R^d$, satisfying
\[
\rho^{t,x}(s,y)>0,\qquad ds\otimes dy\text{-a.e. on }(t,T)\times\mathbb R^d.
\]
As a consequence, given $\mu\in\Pi$ and a sampling grid $G$, for every nonnegative Borel function
$h:(t,T)\times\mathbb R^d\times A\to[0,+\infty]$, it holds that
\begin{equation}
\label{equation:sampled-null-equivalence}
\mathbb E^{\mathbb Q^{\mu,G}}\left[\int_t^T h(s,X_s^{t,x,\mu,G},\alpha_s^{\mu,G})\,ds\right]=0
\quad\Longleftrightarrow\quad
h=0\quad ds\otimes dy\otimes da\text{-a.e.}\text{ on }(t,T)\times\mathbb R^d\times A.
\end{equation}
The same equivalence holds under $\mathbb P$.
\end{proposition}
\begin{proof}
Set $\Sigma(t,x)\coloneqq\tfrac12\sigma(t,x)\sigma(t,x)^\top$. Under Assumption~\ref{assumption:coefficients}, this $d\times d$ matrix is locally bounded, locally Lipschitz
in the state uniformly in time, and locally uniformly positive definite. The reference drift $b_0$ is locally bounded. In the divergence-form formulation, the drift $b_0-\operatorname{div}\Sigma$ is locally essentially bounded as well. Thus \cite[Corollaries~2.2 and~3.1]{BogachevRocknerShaposhnikov2009}
give strictly positive and continuous transition densities for the reference diffusion. Since part of the analysis in \cite{BogachevRocknerShaposhnikov2009} is carried out for $d>2$, in the cases $d=1$ and $d=2$ these results are applied to the three-dimensional diffusion $(X_s^{t,x},B'_s-B'_t)_{s\in[t,T]}$, where $B'$ is a $(3-d)$-dimensional Brownian motion independent of $B$, defined on a possibly enlarged probability space. Its coefficients satisfy the same local conditions and, by independence, its density at time $s$ is the product of the density of $X_s^{t,x}$ and of a Gaussian density, so that the conclusions transfer to $X_s^{t,x}$.

\noindent Under $\mathbb Q^{\mu,G}$, the process $X^{t,x,\mu,G}$ is adapted to $\mathbb F^{B^{\mu,G}}$ and the randomizations are independent of $B^{\mu,G}$ by Remark~\ref{remark:reference-randomization}. For $s\in[t_i,t_{i+1})$, the conditional law of $\alpha^{\mu,G}_s$ given $\mathcal F_s^{B^{\mu,G}}$ therefore has density $\mu(a\,|\,t_i,X^{t,x,\mu,G}_{t_i})$. Consequently,
\[
 \mathbb E^{\mathbb Q^{\mu,G}}\left[\int_t^T h(s,X^{t,x,\mu,G}_s,\alpha^{\mu,G}_s)\,ds\right]
 =\sum_{i=0}^{N_G-1}\mathbb E^{\mathbb Q^{\mu,G}}
   \left[\int_{t_i}^{t_{i+1}}\int_A
   h(s,X^{t,x,\mu,G}_s,a)\mu(a\,|\,t_i,X^{t,x,\mu,G}_{t_i})\,da\,ds\right].
\]
Since the policy density is strictly positive, this expression
vanishes if and only if $h(s,X^{t,x,\mu,G}_s,a)=0$
$ds\otimes d\mathbb Q^{\mu,G}\otimes da$-almost everywhere. It follows that \eqref{equation:sampled-null-equivalence} holds. Equivalence of $\mathbb P$ and $\mathbb Q^{\mu,G}$ proves the last assertion.
\end{proof}

\subsection{A BSDE representation for the policy value}

We now provide a BSDE representation of the value $J^\pi$ associated with a fixed policy $\pi\in\Pi$. This is the first step towards the BSDE characterization of the $q$-function developed below.

\medskip\noindent Fix $\pi\in\Pi$ and define
\begin{equation}\label{equation:policy-driver}
    f_\pi(t,x,y,z)\coloneqq\langle z,\bar\psi_\pi(t,x)\rangle +R_\pi(t,x)
    -\beta y,\qquad(t,x,y,z)\in [0,T]\times\mathbb R^d\times\mathbb R\times\mathbb R^{n}.
\end{equation}

\noindent We first prove the following elementary properties of $f_\pi$.

\begin{lemma}\label{lemma:policy-driver}
For every $\pi\in\Pi$, the driver $f_\pi$ is Borel measurable and uniformly
Lipschitz continuous in $(y,z)$. More precisely,
\begin{equation}\label{equation:policy-driver-lipschitz}
    |f_\pi(t,x,y,z)-f_\pi(t,x,y',z')|
    \leq
    \beta |y-y'|
    +
    \|\psi\|_\infty |z-z'|
\end{equation}
for every $(t,x)\in[0,T]\times\mathbb R^d$,
$y,y'\in\mathbb R$, and $z,z'\in\mathbb R^n$. Moreover,
\begin{equation}\label{equation:policy-driver-monotonicity}
   (y-y')\bigl(f_\pi(t,x,y,z)-f_\pi(t,x,y',z)\bigr)= -\beta |y-y'|^2,
\end{equation}
and there exist constants $C_\pi>0$ and $m_\pi\geq1$ such that
\begin{equation}\label{equation:policy-driver-growth}
    |f_\pi(t,x,0,0)|
    \leq
    C_\pi(1+|x|^{m_\pi})
\end{equation}
for every $(t,x)\in[0,T]\times\mathbb R^d$.
\end{lemma}

\begin{proof}
Measurability follows from the measurability of
$\bar\psi_\pi$ and $R_\pi$. Since $f_\pi$ is affine in $(y,z)$ and $|\bar\psi_\pi|\leq\|\psi\|_\infty$, both \eqref{equation:policy-driver-lipschitz} and \eqref{equation:policy-driver-monotonicity} are immediate. Finally,
\eqref{equation:policy-driver-growth} follows from the polynomial growth of $R_\pi$.
\end{proof}

\noindent Let now $\mathbb Q$ be a probability measure under which $B$ is an $n$-dimensional Brownian motion, and let $X^{t,x}$ denote the solution of the reference equation \eqref{equation:reference-diff}. We consider
the BSDE
\begin{equation}\label{equation:policy-bsde}
    Y_s^{t,x,\pi}=g(X_T^{t,x})+\int_s^T
    f_\pi\bigl(v,X_v^{t,x},Y_v^{t,x,\pi}, Z_v^{t,x,\pi}\bigr)\,dv-\int_s^T \langle
    Z_v^{t,x,\pi},dB_v\rangle,\qquad s\in[t,T].
\end{equation}

\begin{proposition}
\label{proposition:bsde-value}
Let $\pi\in\Pi$ and $(t,x)\in[0,T]\times\mathbb R^d$. Then, for every $p>1$, \eqref{equation:policy-bsde} admits a unique solution $(Y^{t,x,\pi},Z^{t,x,\pi})\in\mathcal S^p_{t,T}
    (\mathbb R;\mathbb Q,\mathbb F^B)\times
    \mathcal H^p_{t,T}(\mathbb R^{n};\mathbb Q,\mathbb F^B)$. Moreover, there exist constants
$\widetilde C_{\pi,p}>0$ and $\widetilde m_{\pi,p}\geq1$ such that
\begin{equation}\label{equation:bsde-moments}
    \mathbb E^{\mathbb Q}\Bigg[\sup_{t\leq s\leq T}
    |Y_s^{t,x,\pi}|^p+\left(\int_t^T
    |Z_s^{t,x,\pi}|^2\,ds\right)^{p/2}\Bigg]
    \leq \widetilde C_{\pi,p}\bigl(1+|x|^{\widetilde m_{\pi,p}}\bigr).
\end{equation}
Finally, the policy value $J^\pi$ is Borel measurable and satisfies
\begin{equation}\label{equation:value-representation}
    Y_s^{t,x,\pi}=J^\pi(s,X_s^{t,x}),\qquad\mathbb Q\text{-a.s.},\;
    s\in[t,T].
\end{equation}
In particular, $Y_t^{t,x,\pi}=J^\pi(t,x)$, and this quantity is deterministic and independent of the particular Brownian probability space used to realize the reference diffusion $X^{t,x}$.
\end{proposition}

\begin{proof}
We divide the proof into three steps.

\medskip
\noindent
\emph{Step 1: existence and uniqueness of $(Y^{t,x,\pi},Z^{t,x,\pi})$, and moment estimates.}\\
Fix $p>1$ and set $F_s^\pi(y,z)
    \coloneqq
    f_\pi(s,X_s^{t,x},y,z)$. Since $X^{t,x}$ is continuous and $\mathbb F^B$-adapted and, by Lemma~\ref{lemma:policy-driver}, $f_\pi$ is Borel measurable, the process $F^\pi(y,z)$ is
$\mathbb F^B$-progressively measurable for every $(y,z)$. We now verify the assumptions of
\cite[Theorem~4.2]{BriandDelyonHuPardouxStoica2003}. By Assumption~\ref{assumption:coefficients} and the moment estimates for the reference diffusion,
\[
    \mathbb E^{\mathbb Q}\left[|g(X_T^{t,x})|^p
    \right]\leq C_p(1+|x|^{mp}).
\]
Moreover, by \eqref{equation:policy-driver-growth},
\begin{align*}
    \mathbb E^{\mathbb Q}
    \left[\left(\int_t^T|F_s^\pi(0,0)|\,ds\right)^p
    \right]\leq \widetilde C_{\pi,p}\mathbb E^{\mathbb Q}
    \left[\left(\int_t^T(1+|X_s^{t,x}|^{m_\pi})\,ds
    \right)^p\right]\leq \widetilde C_{\pi,p}(1+|x|^{\widetilde m_{\pi,p}}),
\end{align*}
so condition \textnormal{(H1)} holds.

\noindent Conditions \textnormal{(H2)--(H4)} follow from
Lemma~\ref{lemma:policy-driver}. Finally, for every $R>0$,
\[
    \sup_{|y|\leq R}|F_s^\pi(y,0)-F_s^\pi(0,0)|
    =\beta R,
\]
and hence condition \textnormal{(H5)} is satisfied.

\noindent Therefore, by
\cite[Theorem~4.2]{BriandDelyonHuPardouxStoica2003},
\eqref{equation:policy-bsde} admits a unique solution in $\mathcal S^p_{t,T}(\mathbb R;\mathbb Q,\mathbb F^B)\times\mathcal H^p_{t,T}(\mathbb R^{n};\mathbb Q,\mathbb F^B)$. The a priori estimate of \cite[Proposition~3.2]{BriandDelyonHuPardouxStoica2003}, together with the preceding bounds, yields
\eqref{equation:bsde-moments}.

\medskip
\noindent
\emph{Step 2: conditional representation of $J^\pi$.}\\
Since $|\bar\psi_\pi(t,x)|\leq \|\psi\|_\infty$, Novikov's condition is satisfied. We may therefore define an equivalent probability measure $\widetilde{\mathbb Q}^{\,\pi}$ by
\begin{equation*}
    \frac{d\widetilde{\mathbb Q}^{\,\pi}}{d\mathbb Q}=\exp\left(\int_t^T\langle\bar\psi_\pi(v,X_v^{t,x}),dB_v\rangle-\frac12\int_t^T |\bar\psi_\pi(v,X_v^{t,x})|^2\,dv
    \right).
\end{equation*}
Under $\widetilde{\mathbb Q}^{\,\pi}$, the process
\begin{equation} \label{equation:policy-brownian-motion}
    \widetilde B_s\coloneqq B_s-\int_t^s
    \bar\psi_\pi(v,X_v^{t,x})\,dv,\qquad
    s\in[t,T],
\end{equation}
is a Brownian motion. Since the reference process satisfies
\[
    dX_s^{t,x}=b_0(s,X_s^{t,x})\,ds+\sigma(s,X_s^{t,x})\,dB_s,
\]
substituting $dB_s=d\widetilde B_s+\bar\psi_\pi(s,X_s^{t,x})\,ds$ gives
\[
    dX_s^{t,x}=\Big[b_0(s,X_s^{t,x})+
    \sigma(s,X_s^{t,x})\bar\psi_\pi(s,X_s^{t,x})
    \Big]ds+ \sigma(s,X_s^{t,x})\,d\widetilde B_s.
\]
Thus, under $\widetilde{\mathbb Q}^{\,\pi}$,
$X^{t,x}$ has precisely the exploratory dynamics
\eqref{equation:aggregate-state} associated with $\pi$.

\noindent On the other hand, writing \eqref{equation:policy-bsde} in differential form gives
\[
    dY_s^{t,x,\pi}=-f_\pi\bigl(s,X_s^{t,x},
    Y_s^{t,x,\pi},Z_s^{t,x,\pi}\bigr)\,ds+
    \langle Z_s^{t,x,\pi},dB_s\rangle.
\]
Using \eqref{equation:policy-driver} and
\eqref{equation:policy-brownian-motion}, the terms involving
$\bar\psi_\pi$ cancel and we obtain
\[
    dY_s^{t,x,\pi}
    =
    \bigl[
        \beta Y_s^{t,x,\pi}
        -
        R_\pi(s,X_s^{t,x})
    \bigr]\,ds
    +
    \langle
        Z_s^{t,x,\pi},
        d\widetilde B_s
    \rangle.
\]
The moment estimate \eqref{equation:bsde-moments}, together with the finite moments of the Girsanov density and its inverse, guarantees the required integrability under $\widetilde{\mathbb Q}^{\,\pi}$.
Applying It\^o's formula to $s\mapsto e^{-\beta(s-t)}Y_s^{t,x,\pi}$ and taking conditional
expectations therefore yields, for every $s\in[t,T]$,
\begin{equation}\label{equation:conditional-policy-value}
    Y_s^{t,x,\pi}=\mathbb E^{\widetilde{\mathbb Q}^{\,\pi}}\bigg[\int_s^T e^{-\beta(v-s)}
    R_\pi(v,X_v^{t,x})\,dv+e^{-\beta(T-s)} g(X_T^{t,x})\,\Big|\,\mathcal F_s^B \bigg].
\end{equation}

\medskip\noindent\emph{Step 3: identification with the policy value.}\\
We first take $s=t$ in \eqref{equation:conditional-policy-value}. The random variable inside the conditional expectation is a functional of the increments $(B_v-B_t)_{t\leq v\leq T}$, and the same holds for the density of $\widetilde{\mathbb Q}^{\,\pi}$ with respect to $\mathbb Q$. This random variable is therefore independent of $\mathcal F_t^B$ under $\widetilde{\mathbb Q}^{\,\pi}$. Since $(X^{t,x},\widetilde B)$ is a weak solution of \eqref{equation:aggregate-state} under $\widetilde{\mathbb Q}^{\,\pi}$, we obtain
\[
    Y_t^{t,x,\pi}
    =
    \mathbb E^{\widetilde{\mathbb Q}^{\,\pi}}\bigg[\int_t^T e^{-\beta(v-t)}
    R_\pi(v,X_v^{t,x})\,dv+e^{-\beta(T-t)} g(X_T^{t,x})\bigg]
    =
    J^\pi(t,x).
\]
It remains to prove \eqref{equation:value-representation} for $s>t$. The driver $f_\pi$ is Borel measurable and uniformly Lipschitz continuous in $(y,z)$, while $f_\pi(\cdot,\cdot,0,0)$ and $g$ have polynomial growth. By the Markov property of solutions of BSDEs \cite[Theorem~4.1]{ElKarouiPengQuenez1997}, there exists a Borel measurable function $u_\pi:[0,T]\times\mathbb R^d\to\mathbb R$ such that, for every initial condition $(t,x)$,
\[
    Y_s^{t,x,\pi}=u_\pi(s,X_s^{t,x}),
    \qquad \mathbb Q\text{-a.s.},\; s\in[t,T].
\]
Taking $s=t$ and using the first part of this step, we obtain $u_\pi(t,x)=Y_t^{t,x,\pi}=J^\pi(t,x)$ for every $(t,x)\in[0,T]\times\mathbb R^d$. Hence $J^\pi=u_\pi$ is Borel measurable and \eqref{equation:value-representation} holds.
\end{proof}

\noindent In particular, the representation above may be applied on the
reference space $(\mathbb Q,B)=(\mathbb Q^{\mu,G},B^{\mu,G})$ associated with any behavior policy $\mu\in\Pi$ and sampling grid
$G$. By Section~\ref{section:reference-dynamics}, the corresponding reference diffusion is $X^{t,x,\mu,G}$. Hence
\begin{equation}\label{equation:observed-value-representation}
    Y_s^{t,x,\pi}
    =
    J^\pi(s,X_s^{t,x,\mu,G}),
    \qquad \mathbb Q^{\mu,G}\text{-a.s.},\;
    s\in[t,T].
\end{equation}

\medskip\noindent Moreover, \eqref{equation:density-moments},
\eqref{equation:bsde-moments}, and H\"older's inequality transfer
the BSDE estimates from $\mathbb Q^{\mu,G}$ to $\mathbb P$. Consequently, for every $p>1$, $(Y^{t,x,\pi},Z^{t,x,\pi})
    \in
    \mathcal S^p_{t,T}
    (\mathbb R;\mathbb P,\mathbb F^G)
    \times
    \mathcal H^p_{t,T}
    (\mathbb R^{n};\mathbb P,\mathbb F^G)$.

\subsection{Continuity of policy values}
\label{section:continuity-of-values}

The coefficients $\bar\psi_\pi$ and $R_\pi$ of the policy driver, as well as the terminal reward $g$, are merely Borel measurable. Nevertheless, the policy value $J^\pi$ is continuous. This regularizing effect is due to the nondegeneracy of $\sigma$. It will be used in Sections~\ref{section:policy-evaluation} and~\ref{section:optimal-value-and-q} to extend to every time and state identities that are first obtained almost everywhere. Since the same property is needed for the optimal BSDE of Section~\ref{section:optimal-value-and-q}, we establish it for a general linear driver.

\medskip\noindent Throughout this subsection, $(\mathbb Q,B)$ is a Brownian probability space as in Section~\ref{section:reference-dynamics}. We extend the reference diffusion to the whole interval $[0,T]$ by setting $X_s^{t,x}\coloneqq x$ for $s\in[0,t]$, and we write $\mathcal O_R\coloneqq\{y\in\mathbb R^d:|y|<R\}$ for $R>0$. We first recall two classical estimates. By \eqref{equation:coefficient-lipschitz} and \eqref{equation:coefficient-growth}, for every $p\geq2$ there exists $C_p>0$ such that, for all $(t,x),(t',x')\in[0,T]\times\mathbb R^d$,
\begin{equation}\label{equation:flow-estimate}
\begin{aligned}
    \mathbb E^{\mathbb Q}\bigg[\sup_{0\leq s\leq T}|X_s^{t,x}|^p\bigg]
    &\leq C_p(1+|x|^p),\\
    \mathbb E^{\mathbb Q}\bigg[\sup_{0\leq s\leq T}
    \bigl|X_s^{t,x}-X_s^{t',x'}\bigr|^2\bigg]
    &\leq C_2\bigl(1+|x|^2+|x'|^2\bigr)
    \bigl(|x-x'|^2+|t-t'|\bigr).
\end{aligned}
\end{equation}
Furthermore, we recall the following estimate for nondegenerate It\^o processes with measurable coefficients \cite[Theorem 4 of Section 2, Chapter~2]{Krylov1980}. For $R>0$, let
\[
    \tau_R^{t,x}\coloneqq\inf\{s\geq t:|X_s^{t,x}|\geq R\}\wedge T.
\]
Up to time $\tau_R^{t,x}$, the drift and the diffusion coefficient of the reference diffusion are bounded and, by \eqref{equation:ellipticity}, its diffusion matrix is uniformly positive definite. Hence there exists a constant $K_R>0$, depending only on $d$, $T$, $R$, $\lambda_R$, and the constant in \eqref{equation:coefficient-growth}, such that, for every initial condition $(t,x)$ and every Borel measurable function $\phi:[0,T]\times\mathbb R^d\to[0,+\infty]$,
\begin{equation}\label{equation:krylov-estimate}
    \mathbb E^{\mathbb Q}\bigg[\int_t^{\tau_R^{t,x}}\phi(s,X_s^{t,x})\,ds\bigg]
    \leq K_R\,\|\phi\|_{L^{d+1}([0,T]\times\mathcal O_R)}.
\end{equation}
The uniformity of this bound with respect to the initial condition is the key ingredient of the following lemma.

\begin{lemma}\label{lemma:path-functionals}
Let $\vartheta$ be a Borel measurable function on $[0,T]\times\mathbb R^d$, with values in a Euclidean space and with polynomial growth. If $(t_j,x_j)\to(t,x)$ in $[0,T]\times\mathbb R^d$, then, for every $p\geq1$,
\[
    \mathbb E^{\mathbb Q}\left[\left(\int_0^T
    \Bigl|\vartheta(s,X_s^{t_j,x_j})\mathbbm 1_{[t_j,T]}(s)
    -\vartheta(s,X_s^{t,x})\mathbbm 1_{[t,T]}(s)\Bigr|^2ds\right)^p\right]
    \longrightarrow0.
\]
\end{lemma}

\begin{proof}
Write $X^j\coloneqq X^{t_j,x_j}$ and $X\coloneqq X^{t,x}$, and 
\[
    D_j\coloneqq\int_0^T
    \Bigl|\vartheta(s,X_s^{j})\mathbbm 1_{[t_j,T]}(s)
    -\vartheta(s,X_s)\mathbbm 1_{[t,T]}(s)\Bigr|^2ds.
\]
By the polynomial growth of $\vartheta$ and the first estimate in \eqref{equation:flow-estimate}, the sequence $(D_j)_j$ is bounded in $L^{2p}(\mathbb Q)$ for every $p\geq 1$. It therefore suffices to prove that $D_j\to0$ in probability.

\noindent Let $\varepsilon>0$ and fix $R>\max\{|x|,\sup_j|x_j|\}$. Since $\vartheta$ is bounded on $[0,T]\times\mathcal O_R$, there exists a bounded and uniformly continuous function $\vartheta^\varepsilon$ on $[0,T]\times\mathbb R^d$ such that $\|\vartheta-\vartheta^\varepsilon\|_{L^{2(d+1)}([0,T]\times\mathcal O_R)}
    \leq\varepsilon$ (see, for instance, \cite[Theorem 3.14]{Rudin}).
Now, consider the event $\Omega_R^j\coloneqq
    \big\{\sup_{0\leq s\leq T}|X_s^j|<R\big\}
    \cap
    \big\{\sup_{0\leq s\leq T}|X_s|<R\big\}$. On such an event it holds that $\tau_R^{t_j,x_j}=\tau_R^{t,x}=T$ and, in turn,
\begin{equation}\label{equation:D-j}
    D_j\leq3\left[
    \int_{t_j}^{\tau_R^{t_j,x_j}}|\vartheta-\vartheta^\varepsilon|^2(s,X_s^j)\,ds
    +\int_t^{\tau_R^{t,x}}|\vartheta-\vartheta^\varepsilon|^2(s,X_s)\,ds
    +E_j^\varepsilon\right],
\end{equation}
where
\begin{align*}
    E_j^\varepsilon
    &\coloneqq\int_0^T
    \Bigl|\vartheta^\varepsilon(s,X_s^j)\mathbbm 1_{[t_j,T]}(s)
    -\vartheta^\varepsilon(s,X_s)\mathbbm 1_{[t,T]}(s)\Bigr|^2ds\\
    &\leq2\|\vartheta^\varepsilon\|_\infty^2|t_j-t|
    +2\int_0^T\bigl|\vartheta^\varepsilon(s,X_s^j)-\vartheta^\varepsilon(s,X_s)\bigr|^2ds.
\end{align*}
By \eqref{equation:krylov-estimate}, applied with $\phi=|\vartheta-\vartheta^\varepsilon|^2$, for which $\|\phi\|_{L^{d+1}([0,T]\times\mathcal O_R)}=\|\vartheta-\vartheta^\varepsilon\|^2_{L^{2(d+1)}([0,T]\times\mathcal O_R)}\leq\varepsilon^2$, each of the first two integrals in \eqref{equation:D-j} has expectation at most $K_R\varepsilon^2$. By \eqref{equation:flow-estimate}, $\sup_{0\leq s\leq T}|X_s^j-X_s|$ vanishes in probability; thus, since $\vartheta^\varepsilon$ is bounded and uniformly continuous, $E_j^\varepsilon\to0$ in probability as $j\to\infty$. Moreover, by \eqref{equation:flow-estimate} and Markov's inequality,
\begin{equation}\label{equation:omega-R}
    \mathbb Q\bigl((\Omega_R^j)^c\bigr)\le\frac{1}{R^2}
    \bigg(\mathbb E^{\mathbb Q}\Big[\sup_{0\le s\le T}|X_s^j|^2\Big]+\mathbb E^{\mathbb Q}\Big[\sup_{0\le
    s\le T}|X_s|^2\Big]\bigg)\le \frac{C}{R^2},
\end{equation}
where $C$ is independent of $j$ and $R$, since $(x_j)_j$ is bounded.

\noindent Fix $\eta>0$. On $\Omega_R^j$, inequality \eqref{equation:D-j} implies that $D_j\le 9\eta$ whenever each of the three terms in brackets is at most $\eta$. Therefore, the union bound and Markov's inequality,
together with \eqref{equation:omega-R}, yield
\[
    \mathbb Q(D_j>9\eta)\le\frac{C}{R^2}
    +\frac{2K_R\varepsilon^2}{\eta}
    +\mathbb Q(E_j^\varepsilon>\eta).
\]
For fixed $R$ and $\varepsilon$, the convergence of
$E_j^\varepsilon$ to zero in probability gives
\[
    \limsup_{j\to\infty}\mathbb Q(D_j>9\eta)\le\frac{C}{R^2}
    +\frac{2K_R\varepsilon^2}{\eta}.
\]
Letting first $\varepsilon\downarrow0$ and then $R\uparrow+\infty$ concludes the proof.
\end{proof}

\begin{proposition}\label{proposition:continuity}
Let $\zeta:[0,T]\times\mathbb R^d\to\mathbb R^n$ be bounded and Borel measurable, and let $h:[0,T]\times\mathbb R^d\to\mathbb R$ be Borel measurable with polynomial growth. For $(t,x)\in[0,T]\times\mathbb R^d$, let $(Y^{t,x},Z^{t,x})$ be the solution of
\begin{equation}\label{equation:linear-bsde}
    Y_s=g(X_T^{t,x})+\int_s^T
    \Bigl[\langle Z_v,\zeta(v,X_v^{t,x})\rangle+h(v,X_v^{t,x})-\beta Y_v\Bigr]\,dv
    -\int_s^T\langle Z_v,dB_v\rangle,
    \qquad s\in[t,T],
\end{equation}
which is unique in $\mathcal S^p_{t,T}(\mathbb R;\mathbb Q,\mathbb F^B)\times\mathcal H^p_{t,T}(\mathbb R^n;\mathbb Q,\mathbb F^B)$ for every $p>1$, and set $w(t,x)\coloneqq Y_t^{t,x}$. Then $w$ is a deterministic Borel measurable function with polynomial growth, $w(T,\cdot)=g$, and
\begin{equation}\label{equation:linear-markov}
    Y_s^{t,x}=w(s,X_s^{t,x}),
    \qquad \mathbb Q\text{-a.s.},\; s\in[t,T].
\end{equation}
Moreover, $w$ is continuous on $[0,T)\times\mathbb R^d$, and on $[0,T]\times\mathbb R^d$ if $g$ is continuous. In particular, for every $\pi\in\Pi$, the policy value $J^\pi$ is continuous on $[0,T)\times\mathbb R^d$, and on $[0,T]\times\mathbb R^d$ if $g$ is continuous.
\end{proposition}

\begin{proof}
\emph{Step 1: well-posedness and representation.}\\
Lemma~\ref{lemma:policy-driver} and the proof of Proposition~\ref{proposition:bsde-value}  apply with
$(\bar\psi_\pi,R_\pi)$ replaced by $(\zeta,h)$.
Thus \eqref{equation:linear-bsde} is well posed, $w$ is deterministic, Borel measurable, and of polynomial growth, and \eqref{equation:linear-markov} holds. In particular, $w(T,\cdot)=g$.

\noindent For $s\in[t,T]$, define
\[
\mathcal E_s^{t,x}\coloneqq\exp\left(
\int_t^s\langle\zeta(v,X_v^{t,x}),dB_v\rangle
-\frac12\int_t^s|\zeta(v,X_v^{t,x})|^2\,dv
\right).
\]
Since $\zeta$ is bounded, $\mathcal E^{t,x}$ is a strictly positive martingale. Define the probability measure $\widetilde {\mathbb Q}^{t,x}$ by $\frac{d\widetilde {\mathbb Q}^{t,x}}{d \mathbb Q}=\mathcal E_T^{t,x}$. By Girsanov's theorem
\[
    \left(e^{-\beta(s-t)}Y_s^{t,x}+\int_t^s e^{-\beta(v-t)}h(v,X_v^{t,x})\,dv\right)_{s\in[t,T]}
\]
is a martingale under $\widetilde {\mathbb Q}^{t,x}$.
Taking expectations and returning to $\mathbb Q$, we obtain, for every $s\in[t,T]$,
\begin{equation}\label{equation:linear-bsde-flow}
    w(t,x)=\mathbb E^{\mathbb Q}\left[\mathcal E_T^{t,x}\left(e^{-\beta(s-t)}w(s,X_s^{t,x})
    +\int_t^s e^{-\beta(v-t)}h(v,X_v^{t,x})\,dv
    \right)\right].
\end{equation}

\medskip\noindent
\emph{Step 2: continuity of the Girsanov densities.}\\
Let $(t_j,x_j)\to(t,x)$ in $[0,T]\times\mathbb R^d$, write $X^j\coloneqq X^{t_j,x_j}$ and $X\coloneqq X^{t,x}$ as in the proof of Lemma~\ref{lemma:path-functionals}, and set
\[
    \Gamma_s^j\coloneqq\zeta(s,X_s^{t_j,x_j})\mathbbm 1_{[t_j,T]}(s),\qquad\Gamma_s\coloneqq\zeta(s,X_s^{t,x})\mathbbm 1_{[t,T]}(s).
\]
By Lemma~\ref{lemma:path-functionals},
\[
    \mathbb E^{\mathbb Q}\left[\int_0^T|\Gamma_s^j-\Gamma_s|^2\,ds\right]\longrightarrow0.
\]
Combining the It\^o isometry with the Cauchy--Schwarz inequality
and the boundedness of $\zeta$, we obtain
\[
    \|\log\mathcal E_T^{t_j,x_j}
      -\log\mathcal E_T^{t,x}\|_{L^2(\mathbb Q)}
    \le
    \left(1+\|\zeta\|_\infty\sqrt T\right)
    \left(
        \mathbb E^{\mathbb Q}\left[
            \int_0^T|\Gamma_s^j-\Gamma_s|^2\,ds
        \right]
    \right)^{1/2}
    \longrightarrow0.
\]
By continuity of the exponential function,
$\mathcal E_T^{t_j,x_j}\to\mathcal E_T^{t,x}$ in probability.
Moreover, the boundedness of $\zeta$ gives, for every $p\ge1$,
\[
    \sup_{(t,x)\in[0,T]\times\mathbb R^d}
    \mathbb E^{\mathbb Q}\bigl[(\mathcal E_T^{t,x})^p\bigr]
    \le
    \exp\left(\frac{p(p-1)}2\|\zeta\|_\infty^2T\right).
\]
Taking $p=3$, we deduce that
$\bigl(|\mathcal E_T^{t_j,x_j}-\mathcal E_T^{t,x}|^2\bigr)_j$
is uniformly integrable. Vitali's convergence theorem therefore
yields
\[
    \mathcal E_T^{t_j,x_j}
    \longrightarrow
    \mathcal E_T^{t,x}
    \qquad\text{in }L^2(\mathbb Q).
\]

\medskip\noindent
\emph{Step 3: continuity before the horizon.}\\
We first consider an arbitrary Borel measurable function $\widetilde h:[0,T]\times\mathbb R^d\to\mathbb R$ of polynomial growth. Lemma~\ref{lemma:path-functionals}, applied to
$(s,y)\mapsto e^{-\beta s}\widetilde h(s,y)$, gives
\[
    \mathbb E^{\mathbb Q}\left[\int_0^T|e^{-\beta s}\widetilde h(s,X^j_s)\mathbbm 1_{[t_j,T]}(s)-e^{-\beta s}\widetilde h(s,X_s)\mathbbm 1_{[t,T]}(s)|^2\,ds
    \right]\longrightarrow0.
\]
By the Cauchy--Schwarz inequality, it follows that
\[
    \int_{t_j}^T e^{-\beta(s-t_j)}\widetilde h(s,X_s^j)\,ds\longrightarrow\int_t^T e^{-\beta(s-t)}\widetilde h(s,X_s)\,ds\qquad\text{in }L^2(\mathbb Q).
\]
Together with the convergence established in Step~2, this proves that the map
\[
    (t,x)\longmapsto\mathbb E^{\mathbb Q}\left[
    \mathcal E_T^{t,x}\int_t^T e^{-\beta(s-t)}
\widetilde h(s,X_s^{t,x})\,ds\right]
\]
is continuous on $[0,T]\times\mathbb R^d$.
We now express $w$ in this form by averaging
\eqref{equation:linear-bsde-flow} over the intermediate time. Fix $\delta\in(0,T/2)$ and let $t\in[0,T-2\delta)$. Integrating \eqref{equation:linear-bsde-flow} over $s\in[T-2\delta,T-\delta]$
and dividing by $\delta$, we obtain
\[
    w(t,x)=\mathbb E^{\mathbb Q}\left[\mathcal E_T^{t,x}\int_t^T e^{-\beta(s-t)}h_\delta(s,X_s^{t,x})\,ds\right],
\]
where
\[
    h_\delta(s,y)\coloneqq\frac1\delta\mathbbm 1_{[T-2\delta,T-\delta]}(s)w(s,y)+\min\left\{1,\frac{(T-\delta-s)^+}{\delta}\right\}h(s,y).
\]
For fixed $\delta$, the function $h_\delta$ is Borel measurable and of polynomial growth. Therefore $w$ is continuous on $[0,T-2\delta)\times\mathbb R^d$.
By the arbitrariness of $\delta$, $w$ is continuous on $[0,T)\times\mathbb R^d$.

\medskip\noindent
\emph{Step 4: continuity at the horizon.}\\
Assume that $g$ is continuous, and let $(t_j,x_j)\to(T,x)$.
By \eqref{equation:flow-estimate}, $X_T^{t_j,x_j}\longrightarrow x$ in probability, and the moments of $X_T^{t_j,x_j}$ of every order are bounded uniformly in $j$. The continuity and polynomial growth of $g$ therefore imply
\[
    g(X_T^{t_j,x_j})\longrightarrow g(x)
    \qquad\text{in }L^2(\mathbb Q).
\]
Moreover, Steps~2 and~3 give
\[
    \mathcal E_T^{t_j,x_j}\longrightarrow
    \mathcal E_T^{T,x}=1,\qquad\int_{t_j}^T e^{-\beta (s-t_j)}h(s,X^j_s)\,ds \longrightarrow\int_{T}^T e^{-\beta (s-T)}h(s,X_s)\,ds =0\qquad\text{in }L^2(\mathbb Q).
\]
Taking $s=T$ in \eqref{equation:linear-bsde-flow}, we obtain
\[
    w(t_j,x_j)=\mathbb E^{\mathbb Q}\left[\mathcal E_T^{t_j,x_j}
    \left(e^{-\beta(T-t_j)}g(X_T^{t_j,x_j})
    +\int_{t_j}^T e^{-\beta (s-t_j)}h(s,X^j_s)\,ds 
    \right)\right]\longrightarrow g(x)=w(T,x).
\]
\noindent  Finally, let $\pi\in\Pi$. For $(\zeta,h)=(\bar\psi_\pi,R_\pi)$, equation \eqref{equation:linear-bsde} is the policy equation \eqref{equation:policy-bsde}, and $w=J^\pi$ by Proposition~\ref{proposition:bsde-value}. This concludes the proof.
\end{proof}

\subsection{Definition and properties of the \texorpdfstring{$q$}{q}-function}

The integrand $Z^{t,x,\pi}$ in the BSDE is defined only up to $ds\otimes d\mathbb Q$-almost everywhere equality and, in particular, its value at the initial time cannot be used directly to define a
deterministic function of $(t,x)$. We therefore first identify a Markovian representative of the family $(Z^{t,x,\pi})_{(t,x)}$ along the reference diffusion.

\begin{proposition}
\label{proposition:markovian-integrand}
For every $\pi\in\Pi$, there exists a Borel measurable function
\[
    z_\pi:[0,T]\times\mathbb R^d\longrightarrow\mathbb R^{n}
\]
such that, for every $(t,x)\in[0,T]\times\mathbb R^d$,
\begin{equation}\label{equation:integrand-representation}
    Z_s^{t,x,\pi}= z_\pi(s,X_s^{t,x}),
    \qquad ds\otimes d\mathbb Q\text{-a.e. on }[t,T]\times\Omega.
\end{equation}
The function $z_\pi$ may be chosen independently of the Brownian probability space used to realize the reference diffusion and is unique up to $dt\otimes dx$-almost everywhere equality.
\end{proposition}
\begin{proof}
The existence of a Borel measurable function $z_\pi$ satisfying \eqref{equation:integrand-representation} follows from \cite[Theorem 4.1]{ElKarouiPengQuenez1997}. Uniqueness up to $dt\otimes dx$-almost everywhere equality follows from \eqref{equation:integrand-representation} and Proposition~\ref{proposition:densities}. More precisely, let $z$ and $z'$ be two Borel measurable functions satisfying \eqref{equation:integrand-representation}. Taking $t=0$ and an arbitrary $x\in\mathbb R^d$, we obtain $z(s,X_s^{0,x})=z'(s,X_s^{0,x})$ $ds\otimes d\mathbb Q$-almost everywhere on $[0,T]\times\Omega$. Since, by Proposition~\ref{proposition:densities}, $X_s^{0,x}$ has a density which is strictly positive $ds\otimes dy$-almost everywhere on $(0,T)\times\mathbb R^d$, we conclude that $z=z'$ $dt\otimes dx$-almost everywhere.
\end{proof}

\noindent We then fix a finite Borel representative of $z_\pi$. Associated with this choice, we define
\begin{equation}\label{equation:q-definition}
    q^\pi(t,x,a)\coloneqq\left\langle z_\pi(t,x),
    \psi(t,x,a)-\bar\psi_\pi(t,x)\right\rangle+
    r(t,x,a)-\bar r_\pi(t,x)+\gamma\ell_\pi(t,x).
\end{equation}
The resulting function is Borel measurable on
$[0,T]\times\mathbb R^d\times A$. If another Borel representative of $z_\pi$ is chosen, the
corresponding $q$-function agrees with $q^\pi$ for every $a\in A$ outside a $dt\otimes dx$-null set.

\medskip\noindent The definition above immediately yields the normalization property satisfied by the $q$-function. This identity already appears in \cite[Theorem~6, equation~(20)]{JiaZhou2023} and \cite[Theorem~6, equation~(6)]{JiaZhou2025Erratum} as part of the martingale characterization of the classical $q$-function. In the present setting, it follows directly from our BSDE-based definition \eqref{equation:q-definition}, and we record it separately since it will play a key role in the policy improvement argument below.

\begin{proposition}\label{proposition:q-normalization}
For every $\pi\in\Pi$ and $(t,x)\in[0,T]\times\mathbb R^d$,
\begin{equation}    \label{equation:q-normalization}
    \int_A\bigl[q^\pi(t,x,a)-\gamma\log\pi(a\,|\, t,x)\bigr]\pi(a\,|\, t,x)\,da=0.
\end{equation}
Moreover, the integral in \eqref{equation:q-normalization} is absolutely convergent.
\end{proposition}

\begin{proof}
By definition of $\bar\psi_\pi$ and $\bar r_\pi$,
\[
    \int_A
    \bigl(
        \psi(t,x,a)-\bar\psi_\pi(t,x)
    \bigr)
    \pi(a\,|\, t,x)\,da
    =
    0,\qquad\int_A
    \bigl(
        r(t,x,a)-\bar r_\pi(t,x)
    \bigr)
    \pi(a\,|\, t,x)\,da
    =
    0.
\]
Substituting these identities, together with the definition of $\ell_\pi$, into \eqref{equation:q-definition} gives \eqref{equation:q-normalization}.
Absolute integrability follows from the boundedness of $\psi$, the polynomial growth of $r$, and
\eqref{equation:policy-moments}--\eqref{equation:entropy-integrability}.
\end{proof}

\begin{remark}\label{remark:classical-q}
The BSDE definition \eqref{equation:q-definition} is consistent with the continuous-time $q$-function introduced in \cite{JiaZhou2023} whenever the policy value is smooth. Indeed, if $J^\pi\in C^{1,2}([0,T)\times\mathbb R^d)$, then one can choose the representative $z_\pi(t,x)=\sigma(t,x)^\top\partial_x J^\pi(t,x)$ for every $(t,x)\in[0,T)\times\mathbb R^d$ and obtain
\begin{equation*}
    q^\pi(t,x,a)
    =
    \bigl(\partial_t+\mathcal L^a-\beta\bigr)
    J^\pi(t,x)
    +
    r(t,x,a),
    \qquad
    a\in A,\quad dt\otimes dx\text{-a.e.\ on }[0,T)\times\mathbb R^d,
\end{equation*}
where
\begin{equation*}
    \mathcal L^a\varphi(t,x)\coloneqq\langle
    b(t,x,a),\partial_x\varphi(t,x)\rangle
    +\frac12\operatorname{tr}\left(
    \sigma(t,x)\sigma(t,x)^\top\partial_x^2\varphi(t,x)
    \right).
\end{equation*}
\end{remark}

\begin{remark}\label{remark:example}
We provide an example of a stochastic optimal control problem satisfying our Assumption~\ref{assumption:coefficients}, but for which the policy value $J^\pi$ is not of class $C^{1,2}$ as required in \cite{JiaZhou2023}. Let $d=n=1$, $A=[-1,1]$, $b_0=0$, $\sigma=1$ and $\psi(t,x,a)=a$, so that
\begin{equation*}
    dX_s=\alpha_s\,ds+dW_s,
\end{equation*}
and, given $\delta>0$, let
\[
    r(t,x,a)=\mathbbm 1_{[-\delta,\delta]}(x)-\frac{a^2}{2},
    \qquad
    g=0,\qquad \beta=0.
\]
Consider the uniform policy $\mu_0(a\,|\,t,x)=1/2$ and its entropy-regularized value $J^{\mu_0}$, defined in \eqref{equation:policy-value}. The associated exploratory state is a Brownian motion, and the averaged reward $R_{\mu_0}$ equals $\mathbbm 1_{[-\delta,\delta]}(x)+\kappa_\gamma$, with $\kappa_\gamma\coloneqq\gamma\log2-1/6$. Denoting by $\Phi$ the standard normal distribution function and writing $\tau\coloneqq T-t$, we obtain
\begin{equation*}
\begin{aligned}
    J^{\mu_0}(t,x)=\int_0^{\tau}\left[
        \Phi\left(\frac{\delta-x}{\sqrt s}\right)
        -\Phi\left(\frac{-\delta-x}{\sqrt s}\right)
    \right]ds+\kappa_\gamma\tau=\Theta(\tau,\delta-x)-\Theta(\tau,-\delta-x)+\kappa_\gamma\tau,
\end{aligned}
\end{equation*}
where $\Theta(s,y)\coloneqq(s+y^2)\,\Phi(y/\sqrt s)+y\sqrt s\,\Phi'(y/\sqrt s)-\max\{y,0\}^2$. The function $J^{\mu_0}$ is continuous, and its spatial derivative
\[
    \partial_xJ^{\mu_0}(t,x)
    =\partial_y\Theta(\tau,x+\delta)-\partial_y\Theta(\tau,x-\delta),
    \qquad
    \partial_y\Theta(\tau,y)
    =2\sqrt\tau\,\Phi'\left(\frac{y}{\sqrt\tau}\right)
    -2|y|\,\Phi\left(-\frac{|y|}{\sqrt\tau}\right),
\]
is bounded and Lipschitz continuous in $x$. However, for $x\neq\pm\delta$, the second-order spatial derivative
\[
    \partial^2_{x}J^{\mu_0}(t,x)
    =2\left[
        \Phi\left(\frac{x+\delta}{\sqrt\tau}\right)
        -\Phi\left(\frac{x-\delta}{\sqrt\tau}\right)
    \right]
    -2\,\mathbbm 1_{(-\delta,\delta)}(x),
\]
jumps at $x=\pm\delta$, for every $t<T$. Hence $J^{\mu_0}$ does not belong to $C^{1,2}([0,T)\times\mathbb R)$. The classical definition of the $q$-function in \cite[Definition~4]{JiaZhou2023} is therefore not available, and $J^{\mu_0}$ is not an admissible candidate value in \cite[Theorem~7]{JiaZhou2025Erratum}. By contrast, the results of this paper apply. Since $\partial_tJ^{\mu_0}$ and $\partial_xJ^{\mu_0}$ are continuous on $[0,T)\times\mathbb R$ and $\partial^2_{x}J^{\mu_0}$ is bounded, It\^o's formula with generalized derivatives \cite[Theorem 1 of Section~10, Chapter~2]{Krylov1980} can be applied to $s\longmapsto J^{\mu_0}(s,X_s^{t,x})$ on $[t,T-\varepsilon]$, for every $\varepsilon\in(0,T-t)$. Comparing the resulting decomposition with \eqref{equation:policy-bsde} shows that one can choose the representative $z_{\mu_0}=\partial_xJ^{\mu_0}$ on $[0,T)\times\mathbb R$; any other representative coincides with it $dt\otimes dx$-almost everywhere. With this choice, the $q$-function defined in \eqref{equation:q-definition} is given by
\begin{equation}\label{equation:q-function-example}
    q^{\mu_0}(t,x,a)
    =a\,\partial_xJ^{\mu_0}(t,x)-\frac{a^2}{2}-\kappa_\gamma.
\end{equation}
Notice that \eqref{equation:q-function-example} involves the value only through its gradient, and hence is well defined everywhere and continuous on $[0,T)\times\mathbb R\times A$.
\end{remark}

\begin{remark}\label{remark:generalized-derivatives}
In the example of Remark~\ref{remark:example}, the value $J^{\mu_0}$ is not of class $C^{1,2}$, but its generalized derivatives are bounded, and this is what allowed us to identify $z_{\mu_0}$ with $\partial_xJ^{\mu_0}$. The same argument applies to any policy $\pi\in\Pi$. Suppose that $J^\pi$ has generalized derivatives $\partial_tJ^\pi$, $\partial_xJ^\pi$, and $\partial_x^2J^\pi$ in $L^{d+1}_{\mathrm{loc}}([0,T)\times\mathbb R^d)$. Then It\^o's formula with generalized derivatives \cite[Theorem 1 of Section~10, Chapter~2]{Krylov1980} can be applied to $s\longmapsto J^\pi(s,X_s^{t,x})$, before the horizon and up to the exit time of $X^{t,x}$ from a ball. Comparing the resulting decomposition with \eqref{equation:policy-bsde}, as in Remark~\ref{remark:example}, we obtain $z_\pi=\sigma^\top\partial_xJ^\pi$ $dt\otimes dx$-almost everywhere, and $q^\pi$ is then given by the formula of Remark~\ref{remark:classical-q}, with the derivatives understood in the generalized sense. Regularity of this kind is the natural one for linear parabolic equations with a nondegenerate diffusion matrix and only measurable lower-order coefficients, such as the equation formally satisfied by $J^\pi$, and it is the object of the $L_p$-theory of such equations; see, e.g., \cite[Chapter~6]{Krylov2008}.

\medskip\noindent
We do not follow this route. The construction based on the BSDE \eqref{equation:policy-bsde} provides $z_\pi$, and hence $q^\pi$, directly under Assumption~\ref{assumption:coefficients}, without any estimate on the derivatives of $J^\pi$. More importantly, generalized derivatives are a property of the value function, not of the functions among which it is sought: in Theorems~\ref{theorem:joint-martingale-characterization} and~\ref{theorem:optimal-martingale-characterization}, the candidate $u$ is an arbitrary continuous function of polynomial growth, to which no version of It\^o's formula applies.
\end{remark}

\section{Policy improvement via the \texorpdfstring{$q$}{q}-function}
\label{section:policy-improvement}

We now show that the $q$-function introduced above can be used to improve a given policy. Our result is the analogue of \cite[Theorem~2]{JiaZhou2023} in the present setting. As in their approach, the updated policy is obtained by applying a Gibbs
transformation to the $q$-function. Here, however, the pointwise improvement induced by the Gibbs update is translated into an ordering of the corresponding BSDE drivers, so that policy improvement follows from the comparison principle for scalar BSDEs.

\subsection{The Gibbs update}

Fix $\pi\in\Pi$. For $(t,x)\in[0,T]\times\mathbb R^d$ and a probability density $\rho$ on $A$, define
\[
    \mathcal G^\pi(t,x;\rho)
    \coloneqq
    \int_A
    \bigl[
        q^\pi(t,x,a)-\gamma\log\rho(a)
    \bigr]\rho(a)\,da,
\]
whenever the integral is well defined. By
\eqref{equation:q-normalization},
\begin{equation}\label{equation:gibbs-functional-normalization}
    \mathcal G^\pi
    \bigl(t,x;\pi(\cdot\,|\,t,x)\bigr)
    =
    0.
\end{equation}
As in \cite{JiaZhou2023}, the policy improvement step is based on the Gibbs density generated by $q^\pi$, formally given by
\begin{equation}\label{equation:gibbs-update}
    \pi'(a\,|\, t,x)
    \coloneqq
    \frac{
        \displaystyle
        \exp\left(
            \frac{q^\pi(t,x,a)}{\gamma}
        \right)
    }{
        \displaystyle
        \int_A
        \exp\left(
            \frac{q^\pi(t,x,a')}{\gamma}
        \right)\,da'
    }.
\end{equation}
For this expression to be well defined, its normalizing constant must be finite. Since $\psi$ is bounded and $z_\pi(t,x)$ is finite, this holds if and only if
\begin{equation}\label{equation:gibbs-integrability}
    \int_A
    \exp\left(
        \frac{r(t,x,a)}{\gamma}
    \right)\,da
    <+\infty.
\end{equation}
The integral in \eqref{equation:gibbs-integrability} is
automatically strictly positive, since $r(t,x,\cdot)$ is finite and $A$ has positive Lebesgue measure. Whenever \eqref{equation:gibbs-integrability} holds, the terms in
$q^\pi(t,x,a)$ that are independent of the action cancel in \eqref{equation:gibbs-update}, and hence
\[
    \pi'(a\,|\, t,x)
    =
    \frac{
        \displaystyle
        \exp\left(
            \frac{
                \langle z_\pi(t,x),\psi(t,x,a)\rangle
                +r(t,x,a)
            }{\gamma}
        \right)
    }{
        \displaystyle
        \int_A
        \exp\left(
            \frac{
                \langle z_\pi(t,x),\psi(t,x,a')\rangle
                +r(t,x,a')
            }{\gamma}
        \right)\,da'
    }.
\]
Notice that condition~\eqref{equation:gibbs-integrability} ensures that the
Gibbs update is well defined as a probability density, but it does not imply that it satisfies all the admissibility requirements of
Definition~\ref{definition:policy}. Accordingly, in what follows we assume that \eqref{equation:gibbs-integrability} holds
$dt\otimes dx$-almost everywhere and that
\eqref{equation:gibbs-update} admits an admissible representative
$\pi'\in\Pi$. Recall from Definition~\ref{definition:policy} that $\pi'$ is admissible if the bounds \eqref{equation:policy-moments} and \eqref{equation:entropy-integrability} hold true. For instance, if $A$ is bounded, then \eqref{equation:gibbs-integrability} holds everywhere and, by \eqref{equation:reward-growth}, there exists a constant $C>0$ such that
\[
    \bigl|\log\pi'(a\,|\, t,x)\bigr|
    \leq
    C\bigl(1+|x|^m+|z_\pi(t,x)|\bigr),
    \qquad (t,x,a)\in[0,T]\times\mathbb R^d\times A.
\]
In this case, the Gibbs update is admissible as soon as $z_\pi$ admits a representative with polynomial growth.

\medskip\noindent The policy improvement argument relies on the following lemma. Before stating it, we notice that, for probability measures $P_1$ and $P_2$ on the same measurable space, we denote by $D_{\mathrm{KL}}(P_1\,\|\,P_2)$ the relative entropy (Kullback--Leibler divergence) of $P_1$ with respect to $P_2$, with $D_{\mathrm{KL}}(P_1\,\|\,P_2)=+\infty$ unless $P_1\ll P_2$.
\begin{lemma}
\label{lemma:gibbs-variational}
Let $\pi\in\Pi$ and suppose that its Gibbs update
$\pi'\in\Pi$ is well defined as above. Then, for $dt\otimes dx$-almost every $(t,x)$ and for every probability density $\rho$ on $A$ for
which $\mathcal G^\pi(t,x;\rho)$ is well defined,
\begin{equation}\label{equation:entropy-variational-identity}
    \mathcal G^\pi
    \bigl(t,x;\rho\bigr)=
    \gamma\log\int_A\exp\left(\frac{q^\pi(t,x,a)}{\gamma}
    \right)\,da-\gamma
    D_{\mathrm{KL}}\left(\rho
    \,\middle\|\,\pi'(\cdot\,|\,t,x)
    \right).
\end{equation}
Consequently, for such $(t,x)$, $\pi'(\cdot\,|\,t,x)$ uniquely maximizes $\mathcal G^\pi(t,x;\cdot)$, up to equality almost everywhere on $A$.
\end{lemma}

\begin{proof}
By \eqref{equation:gibbs-update}, for $dt\otimes dx$-almost every $(t,x)$ and almost every $a\in A$,
\[
    \log\pi'(a\,|\, t,x)=\frac{q^\pi(t,x,a)}{\gamma}
    -\log\int_A\exp\left(\frac{q^\pi(t,x,a')}{\gamma}
    \right)\,da'.
\]
Therefore,
\begin{align*}
    D_{\mathrm{KL}}\left(\rho\,\middle\|\,\pi'(\cdot\,|\,t,x)\right)&=\int_A\rho(a)\log\left(
    \frac{\rho(a)}{\pi'(a\,|\, t,x)}\right)\,da \\
    &=\log\int_A\exp\left(\frac{q^\pi(t,x,a)}{\gamma}\right)\,da-\frac{1}{\gamma}
    \mathcal G^\pi\bigl(t,x;\rho\bigr),
\end{align*}
which gives \eqref{equation:entropy-variational-identity}.
The conclusion follows from the nonnegativity of the
Kullback--Leibler divergence, with equality if and only if $\rho(\cdot)=\pi'(\cdot\,|\,t,x)$ almost everywhere on $A$.
\end{proof}

\noindent In particular, the maximality of $\pi'(\cdot\,|\,t,x)$ stated in
Lemma~\ref{lemma:gibbs-variational} and 
\eqref{equation:gibbs-functional-normalization} give
\begin{equation}
    \mathcal G^\pi
    \bigl(t,x;\pi'(\cdot\,|\,t,x)\bigr)
    \geq
    \mathcal G^\pi
    \bigl(t,x;\pi(\cdot\,|\,t,x)\bigr)
    =
    0,\qquad \text{for $dt\otimes dx$-a.e. $(t,x)$.}
    \label{equation:gibbs-local-improvement}
\end{equation}
Thus the Gibbs update improves the entropy-regularized $q$-criterion pointwise.

\subsection{BSDE policy improvement}

We next show that the pointwise improvement
\eqref{equation:gibbs-local-improvement} translates into an improvement of the policy value. The connection is provided by a simple identity between the entropy-regularized $q$-criterion and the
BSDE drivers.

\medskip\noindent For every $\nu\in\Pi$, $\xi\in\mathbb R$, and for
$ds\otimes dy$-almost every $(s,y)$,
\begin{equation}\label{equation:policy-driver-difference}
    f_\nu(s,y,\xi,z_\pi(s,y))
    -
    f_\pi(s,y,\xi,z_\pi(s,y))
    =
    \mathcal G^\pi
    \bigl(s,y;\nu(\cdot\,|\,s,y)\bigr).
\end{equation}
Indeed,
\[
\begin{aligned}
    f_\nu(s,y,\xi,z_\pi(s,y))-f_\pi(s,y,\xi,z_\pi(s,y))
    =\left\langle z_\pi(s,y),\bar\psi_\nu(s,y)-\bar\psi_\pi(s,y)\right\rangle+ R_\nu(s,y)-R_\pi(s,y).
\end{aligned}
\]
The identity follows directly from \eqref{equation:q-definition} and the definitions of $\bar\psi_\nu$, $\bar r_\nu$, and $\ell_\nu$.

\noindent Notice that the right-hand side in \eqref{equation:policy-driver-difference} is independent of $\xi$. Thus, improvement of the entropy-regularized $q$-criterion is equivalent to an improvement of the BSDE driver when evaluated at the Markovian integrand associated with the current policy $\pi$.

\medskip\noindent Define the optimal value
\[
    J^*(t,x)
    \coloneqq
    \sup_{\nu\in\Pi}J^\nu(t,x).
\]

\begin{theorem}
\label{theorem:policy-improvement}
Let $\pi\in\Pi$ and suppose that its Gibbs update admits an admissible representative $\pi'\in\Pi$. Then
\begin{equation}\label{equation:policy-improvement}
    J^{\pi'}(t,x)\geq J^\pi(t,x),\qquad
    (t,x)\in[0,T]\times\mathbb R^d.
\end{equation}
Moreover, if $\pi$ is a fixed point of the Gibbs update, that is, if $\pi'=\pi$ almost everywhere, then
\begin{equation}\label{equation:gibbs-fixed-point-optimality}
    J^\pi(t,x)=J^*(t,x),\qquad (t,x)\in[0,T]\times\mathbb R^d.
\end{equation}
In particular, $\pi$ is optimal within $\Pi$.
\end{theorem}

\begin{proof}
Fix $(t,x)\in[0,T]\times\mathbb R^d$ and consider the reference diffusion $X^{t,x}$. By \eqref{equation:gibbs-local-improvement} and
\eqref{equation:policy-driver-difference},
\begin{equation*}
    f_{\pi'}(s,y,\xi,z_\pi(s,y))
    \geq
    f_\pi(s,y,\xi,z_\pi(s,y))
\end{equation*}
for every $\xi\in\mathbb R$ and for $ds\otimes dy$-almost every $(s,y)$.

\noindent By Proposition~\ref{proposition:densities}, this $ds\otimes dy$-almost everywhere inequality remains valid when evaluated along the reference diffusion. Together with \eqref{equation:integrand-representation}, this yields \begin{equation*}
    f_{\pi'}\bigl(s,X_s^{t,x},Y_s^{t,x,\pi},
    Z_s^{t,x,\pi}\bigr)\geq f_\pi\bigl(s,X_s^{t,x},
    Y_s^{t,x,\pi},Z_s^{t,x,\pi}\bigr),\qquad ds\otimes d\mathbb Q\text{-a.e.}
\end{equation*}
The BSDEs associated with $\pi'$ and $\pi$ have the same terminal condition $g(X_T^{t,x})$. The comparison theorem for scalar BSDEs (see \cite[Theorem~6.2.2]{Pham2009}) therefore gives $Y_t^{t,x,\pi'}\geq Y_t^{t,x,\pi}$. By Proposition~\ref{proposition:bsde-value}, $Y_t^{t,x,\pi'}=J^{\pi'}(t,x)$ and $Y_t^{t,x,\pi}=J^\pi(t,x)$, which proves \eqref{equation:policy-improvement}.

\medskip\noindent Suppose now that $\pi$ is a fixed point of the Gibbs update. Since $\pi$ is then the maximizer in Lemma~\ref{lemma:gibbs-variational}, for every $\nu\in\Pi$,
\[
    \mathcal G^\pi
    \bigl(s,y;\nu(\cdot\,|\,s,y)\bigr)
    \leq
    \mathcal G^\pi
    \bigl(s,y;\pi(\cdot\,|\,s,y)\bigr)
    =
    0
\]
for $ds\otimes dy$-almost every $(s,y)$. Hence, by
\eqref{equation:policy-driver-difference},
\[
    f_\nu(s,y,\xi,z_\pi(s,y))
    \leq
    f_\pi(s,y,\xi,z_\pi(s,y)),\qquad\text{for all }\xi\in\mathbb R.
\]
Repeating the preceding comparison argument yields $J^\nu(t,x)\leq J^\pi(t,x)$ for every $\nu\in\Pi$. Taking the supremum over $\nu$ gives $J^*(t,x)=J^\pi(t,x)$, and proves \eqref{equation:gibbs-fixed-point-optimality}.
\end{proof}

\section{Martingale characterizations for policy evaluation}
\label{section:policy-evaluation}

We now characterize the policy value $J^\pi$ and the corresponding $q$-function $q^\pi$ through martingale properties of processes constructed from discretely sampled state--action observations.
These results provide the BSDE counterparts of the martingale characterizations in
\cite[Theorems~6 and~7]{JiaZhou2025Erratum}.

\medskip\noindent From now on, $\pi\in\Pi$ denotes the target policy whose
value and $q$-function are to be characterized, and $\mu\in\Pi$ denotes the behavior policy used to generate the observations.

\subsection{The observable process}

Throughout this section and in Section~\ref{section:optimal-value-and-q}, whenever an initial condition $(t,x)\in[0,T)\times\mathbb R^d$, a behavior policy $\mu\in\Pi$, and a sampling grid
\[
    G=\{t=t_0<t_1<\cdots<t_{N_G}=T\}
\]
are fixed, we use the shorthand notations $X\coloneqq X^{t,x,\mu,G}$, $\alpha\coloneqq\alpha^{\mu,G}$.

\medskip\noindent Let $u:[0,T]\times\mathbb R^d\to\mathbb R$ be a Borel measurable function of polynomial growth, continuous on $[0,T)\times\mathbb R^d$, and let $q:[0,T]\times\mathbb R^d\times A\to\mathbb R$ be a Borel function. We associate with $(u,q)$ the observable process
\begin{equation}\label{equation:observable-process}
    M_s^{\mu,G}[u,q]\coloneqq
    e^{-\beta(s-t)}u(s,X_s)+\int_t^s e^{-\beta(v-t)}
    \bigl[r(v,X_v,\alpha_v)-q(v,X_v,\alpha_v)
    \bigr]\,dv,\qquad s\in[t,T].
\end{equation}
The dependence of $M^{\mu,G}[u,q]$ on the initial condition $(t,x)$ is suppressed from the notation. 

\noindent For a given choice of $(t,x,\mu,G)$, we require the candidate $q$ to be integrable along the corresponding sampled trajectory, meaning that
\begin{equation}\label{equation:candidate-integrability}
    \int_t^T
    |q(s,X_s,\alpha_s)|\,ds<+\infty,
    \qquad \mathbb P\text{-a.s.}
\end{equation}
Since $q$ is Borel measurable and $(X,\alpha)$ is progressively measurable, the process $q(\cdot,X_\cdot,\alpha_\cdot)$ is progressively measurable. Together with \eqref{equation:candidate-integrability}, the growth assumptions on $r$, and the moment estimates for $X$ and $\alpha$, this shows that the integral in \eqref{equation:observable-process} is well defined and has absolutely continuous sample paths. Since $X$ is continuous and $u$ is continuous on $[0,T)\times\mathbb R^d$, $M^{\mu,G}[u,q]$ is therefore an $\mathbb F^G$-adapted process with continuous paths on $[t,T)$. Its paths may be discontinuous at the terminal time, where $u(T,\cdot)$ is only assumed to be measurable. For $u=J^\pi$, however, \eqref{equation:observed-value-representation} shows that $M^{\mu,G}[J^\pi,q]$ is continuous on the whole interval $[t,T]$.
 
\medskip\noindent Notice that the true $q$-function $q^\pi$ satisfies \eqref{equation:candidate-integrability} for every initial condition,
behavior policy, and sampling grid. Indeed, by
\eqref{equation:q-definition}, there exist constants
$C_\pi>0$ and $m_\pi\geq1$ such that
\[
    |q^\pi(s,X_s,\alpha_s)|
    \leq
    2\|\psi\|_\infty |z_\pi(s,X_s)|
    +
    C_\pi
    \bigl(
        1+|X_s|^{m_\pi}+|\alpha_s|^{m_\pi}
    \bigr).
\]
Moreover, by Proposition~\ref{proposition:bsde-value} and the integrability transfer established thereafter,
\[
    z_\pi(\cdot,X_\cdot)=Z^{t,x,\pi}\in
    \mathcal H^p_{t,T}(\mathbb R^{n};\mathbb P,\mathbb F^G)\qquad\text{for every }p>1.
\]
In particular, taking $p=2$ and using Cauchy--Schwarz,
\[
    \int_t^T |z_\pi(s,X_s)|\,ds<+\infty,
    \qquad \mathbb P\text{-a.s.}
\]
Together with the moment estimates for $X$ and $\alpha$, this shows that $q^\pi$ satisfies
\eqref{equation:candidate-integrability}.

\begin{remark}
\label{remark:candidate-q-regularity}
In the martingale characterizations of \cite{JiaZhou2023}, the candidate $q$-function is assumed to be continuous. Such a regularity requirement would not be natural in our setting, since the $q$-function $q^\pi$ defined in
\eqref{equation:q-definition} depends on the Markovian BSDE integrand $z_\pi$, which is in general available only as a Borel representative defined up to $dt\otimes dx$-almost everywhere
equality. We therefore work with Borel candidate functions and state explicitly the pathwise integrability condition \eqref{equation:candidate-integrability} needed to ensure that the
observable process \eqref{equation:observable-process} is well defined.
\end{remark}

\subsection{Characterization with a known value function}

We first consider the case in which the value function $J^\pi$ of the target policy is known. The following result corresponds to \cite[Theorem~6]{JiaZhou2025Erratum}. It shows that
$q^\pi$ can be characterized through the martingale property of the observable process introduced above, using trajectories generated by an arbitrary admissible behavior policy $\mu$.

\begin{theorem}
\label{theorem:q-martingale-characterization}
Fix a target policy $\pi\in\Pi$ and a behavior policy $\mu\in\Pi$. Let $q:[0,T]\times\mathbb R^d\times A\longrightarrow\mathbb R$ be Borel measurable and suppose that \eqref{equation:candidate-integrability} holds for almost every initial condition $(t,x)$ and every sampling grid $G$. Then the following statements are equivalent:
\begin{enumerate}[(i)]
    \item $q(t,x,a)=q^\pi(t,x,a)$ $dt\otimes dx\otimes da$-almost everywhere;
    \item for almost every initial condition $(t,x)$ and every sampling grid $G$, the process $M^{\mu,G}[J^\pi,q]$ in \eqref{equation:observable-process} is an $(\mathbb F^G,\mathbb P)$-martingale.
\end{enumerate}
Moreover, if these equivalent conditions hold, then, for every behavior policy $\nu\in\Pi$, every initial condition $(t,x)$, and every sampling grid $G$, the process $M^{\nu,G}[J^\pi,q]$ is an $(\mathbb F^G,\mathbb P)$-martingale.
\end{theorem}

\begin{proof}
Fix an initial condition $(t,x)$ and a sampling grid $G$ for which the assumptions of the theorem hold. Under $\mathbb Q^{\mu,G}$, the process $X$ solves the reference
equation \eqref{equation:reference-diff}, driven by $B^{\mu,G}$. Hence, by
Propositions~\ref{proposition:bsde-value}, \ref{proposition:continuity}, and
\ref{proposition:markovian-integrand},
\begin{align*}
    Y_s^{t,x,\pi}=J^\pi(s,X_s),&\qquad \mathbb Q^{\mu,G}\text{-a.s.},\; s\in[t,T]\\
    Z_s^{t,x,\pi}=z_\pi(s,X_s),&\qquad ds\otimes d\mathbb Q^{\mu,G}\text{-a.e.}
\end{align*}
Therefore, the BSDE \eqref{equation:policy-bsde} yields the semimartingale identity
\begin{align*}
    dJ^\pi(s,X_s)=
    -f_\pi
    \bigl(
        s,X_s,J^\pi(s,X_s),z_\pi(s,X_s)
    \bigr)\,ds+
    \left\langle
        z_\pi(s,X_s),dB_s^{\mu,G}
    \right\rangle .
\end{align*}
Using \eqref{equation:reference-brownian-motion}, it holds that $dB_s^{\mu,G}=dW_s+\psi(s,X_s,\alpha_s)ds$, for $s\in[t,T]$. Then, recalling the definition of $f_\pi$, we obtain
\begin{align*}
    dJ^\pi(s,X_s)&=\Big[
        \beta J^\pi(s,X_s)
        -R_\pi(s,X_s)+
        \left\langle
            z_\pi(s,X_s),
            \psi(s,X_s,\alpha_s)
            -\bar\psi_\pi(s,X_s)
        \right\rangle
    \Big]\,ds\\
    &\qquad+
    \left\langle
        z_\pi(s,X_s),dW_s
    \right\rangle .
\end{align*}
Since $R_\pi=\bar r_\pi-\gamma\ell_\pi$, the definition \eqref{equation:q-definition} of $q^\pi$ gives
\[
    -R_\pi(s,X_s)
    +
    \left\langle
        z_\pi(s,X_s),
        \psi(s,X_s,\alpha_s)-\bar\psi_\pi(s,X_s)
    \right\rangle
    =
    q^\pi(s,X_s,\alpha_s)-r(s,X_s,\alpha_s).
\]
Consequently,
\begin{align}
    d\left(
        e^{-\beta(s-t)}J^\pi(s,X_s)
    \right)
    &=
    e^{-\beta(s-t)}
    \bigl[
        q^\pi(s,X_s,\alpha_s)
        -r(s,X_s,\alpha_s)
    \bigr]\,ds
    \notag\\
    &\qquad+
    e^{-\beta(s-t)}
    \left\langle
        z_\pi(s,X_s),dW_s
    \right\rangle .
    \label{equation:discounted-value-dynamics}
\end{align}
Since $\mathbb Q^{\mu,G}$ and $\mathbb P$ are equivalent, the preceding identity \eqref{equation:discounted-value-dynamics} may be regarded as an identity under $\mathbb P$. Combining \eqref{equation:discounted-value-dynamics} with
\eqref{equation:observable-process}, we obtain
\begin{align}
    M_s^{\mu,G}[J^\pi,q]&=
    J^\pi(t,x)
    +
    \int_t^s
    e^{-\beta(v-t)}
    \left\langle
        z_\pi(v,X_v),dW_v
    \right\rangle
    \notag\\
    &\qquad+
    \int_t^s
    e^{-\beta(v-t)}
    \bigl[
        q^\pi(v,X_v,\alpha_v)
        -q(v,X_v,\alpha_v)
    \bigr]\,dv.
    \label{equation:martingale-decomposition}
\end{align}
Since
$z_\pi(\cdot,X_\cdot)\in
\mathcal H^2_{t,T}(\mathbb R^{n};\mathbb P,\mathbb F^G)$, the stochastic integral above is a square-integrable $(\mathbb F^G,\mathbb P)$-martingale.

\medskip\noindent Now we prove the equivalence \textnormal{(i)} $\Leftrightarrow$ \textnormal{(ii)}.

\medskip
\noindent
\emph{\textnormal{(i)} $\Rightarrow$ \textnormal{(ii)}.}
Suppose that
\[
    q=q^\pi
    \qquad
    dt\otimes dx\otimes da\text{-a.e.}
\]
Applying
Proposition~\ref{proposition:densities} to
$|q-q^\pi|$ gives
\[
    q(s,X_s,\alpha_s)
    =
    q^\pi(s,X_s,\alpha_s)
    \qquad
    ds\otimes d\mathbb P\text{-a.e.}
\]
Hence the finite-variation term in
\eqref{equation:martingale-decomposition} vanishes and
\[
    M_s^{\mu,G}[J^\pi,q]
    =
    J^\pi(t,x)+\int_t^s
    e^{-\beta(v-t)}
    \left\langle
        z_\pi(v,X_v),dW_v
    \right\rangle.
\]
Therefore $M^{\mu,G}[J^\pi,q]$ is an
$(\mathbb F^G,\mathbb P)$-martingale.

\noindent The same argument applies to any behavior policy $\nu\in\Pi$.
Indeed, Proposition~\ref{proposition:densities} yields
\[
    q(s,X_s^{t,x,\nu,G},\alpha_s^{\nu,G})
    =
    q^\pi(s,X_s^{t,x,\nu,G},\alpha_s^{\nu,G}),\qquad\text{$ds\otimes d\mathbb P$-a.e.},
\]
while the 
integrability argument for $q^\pi$ applies under any admissible behavior policy. Hence the process $M^{\nu,G}[J^\pi,q]$ is an $(\mathbb F^G,\mathbb P)$-martingale for every $\nu\in\Pi$, every initial condition, and every sampling grid.

\medskip
\noindent
\emph{\textnormal{(ii)} $\Rightarrow$ \textnormal{(i)}.}
Suppose now that
$M^{\mu,G}[J^\pi,q]$ is an
$(\mathbb F^G,\mathbb P)$-martingale. Define
\[
    \Lambda_s
    \coloneqq
    \int_t^s
    e^{-\beta(v-t)}
    \bigl[
        q^\pi(v,X_v,\alpha_v)
        -q(v,X_v,\alpha_v)
    \bigr]\,dv.
\]
By \eqref{equation:martingale-decomposition},
\[
    \Lambda_s
    =
    M_s^{\mu,G}[J^\pi,q]
    -
    J^\pi(t,x)
    -\int_t^s
    e^{-\beta(v-t)}
    \left\langle
        z_\pi(v,X_v),dW_v
    \right\rangle.
\]
Thus $\Lambda$ is a continuous
$(\mathbb F^G,\mathbb P)$-martingale. On the other hand, the pathwise integrability of $q$ and $q^\pi$ implies that $\Lambda$ has finite variation. Since $\Lambda_t=0$, and $\Lambda$ is a continuous finite-variation martingale,
\[
    \Lambda_s=0,
    \qquad
    s\in[t,T],
    \quad
    \mathbb P\text{-a.s.}
\]
Since $\Lambda$ is absolutely continuous in time, it follows that
\[
    q^\pi(s,X_s,\alpha_s)
    =
    q(s,X_s,\alpha_s)
    \qquad
    ds\otimes d\mathbb P\text{-a.e.}
\]
Applying
Proposition~\ref{proposition:densities} to
$|q-q^\pi|$ now yields
\[
    q^\pi(s,y,a)
    =
    q(s,y,a),\qquad\text{for $ds\otimes dy\otimes da$-a.e. $(s,y,a)\in(t,T)\times\mathbb R^d\times A$.}
\]
It remains only to remove the dependence on the initial time. Since the martingale property in \textnormal{(ii)} holds for almost every initial condition $(t,x)$, Fubini's theorem yields a set $I\subset(0,T)$ of full Lebesgue measure such that,
for every $t\in I$, the martingale property holds for almost every $x\in\mathbb R^d$. Choose a sequence $(t_j)_{j\geq1}\subset I$ with $t_j\downarrow0$ and, for each $j$, an initial state $x_j$ for which the preceding argument applies.
We then obtain
\[
    q=q^\pi
    \qquad
    ds\otimes dy\otimes da\text{-a.e. on }
    (t_j,T)\times\mathbb R^d\times A
\]
for every $j$. Since $\cup_{j\geq1}(t_j,T)=(0,T)$, we conclude that $q=q^\pi$, $dt\otimes dx\otimes da$-a.e., which proves \textnormal{(i)}.
\end{proof}

\begin{remark}
\label{remark:q-characterization-on-off-policy}
Theorem~\ref{theorem:q-martingale-characterization}
contains both the on-policy and off-policy characterizations of
$q^\pi$. Taking $\mu=\pi$ gives the on-policy characterization,
whereas an arbitrary $\mu\in\Pi$ gives the off-policy one. In
particular, the implication
\textnormal{(i)} $\Rightarrow$ \textnormal{(ii)} shows that the
martingale condition is valid under any admissible behavior policy,
while the converse shows that its validity under a single fixed
behavior policy is sufficient to identify $q^\pi$ up to
$dt\otimes dx\otimes da$-almost everywhere equality. These
statements correspond, respectively, to the three parts of
\cite[Theorem~6]{JiaZhou2025Erratum}.
\end{remark}

\subsection{Joint characterization of the value and the \texorpdfstring{$q$}{q}-function}

We now consider the case in which both the policy value and the $q$-function are unknown. In contrast with Theorem~\ref{theorem:q-martingale-characterization}, the martingale condition alone does not identify the pair uniquely. As in \cite[Theorem~7]{JiaZhou2025Erratum}, one must additionally
specify the terminal reward and impose the normalization associated with the target policy. These conditions select, among the pairs compatible with the martingale property, the value and $q$-function corresponding to the prescribed policy-evaluation problem.

\medskip\noindent Our formulation differs from that of \cite{JiaZhou2025Erratum} in that no $C^{1,2}$ regularity is imposed on the candidate value function, which is only required to be continuous before the terminal time. With this requirement, the class of candidates always contains the function to be identified: $J^\pi$ is continuous on $[0,T)\times\mathbb R^d$ by Proposition~\ref{proposition:continuity}. This is not the case when $C^{1,2}$ regularity is imposed, since $J^\pi$ need not be of class $C^{1,2}$ (see Remark~\ref{remark:example}). Since $u$ is only continuous, the proof applies no chain rule to $u(s,X_s)$, not even It\^o's formula with generalized derivatives (see Remark~\ref{remark:generalized-derivatives}).

\begin{theorem}
\label{theorem:joint-martingale-characterization}
Fix a target policy $\pi\in\Pi$ and a behavior policy $\mu\in\Pi$. Let $u:[0,T]\times\mathbb R^d\longrightarrow\mathbb R$ be Borel measurable with polynomial growth and continuous on $[0,T)\times\mathbb R^d$, and let $q:[0,T]\times\mathbb R^d\times A\longrightarrow\mathbb R$ be Borel measurable and suppose that
\eqref{equation:candidate-integrability} holds for almost every initial condition $(t,x)$ and every sampling grid $G$. Suppose that
\begin{align}
    u(T,x)=g(x),&\qquad x\in\mathbb R^d,\label{equation:joint-terminal-condition}\\
    \int_A\bigl[q(t,x,a)-\gamma\log\pi(a\,|\, t,x)
    \bigr]\pi(a\,|\, t,x)\,da
    =0,&\qquad\text{for $dt\otimes dx$-a.e. $(t,x)$},\label{equation:joint-normalization}
\end{align}
where the integral in \eqref{equation:joint-normalization} is assumed to be absolutely
convergent. Then the following statements are equivalent:
\begin{enumerate}[(i)]
    \item $u(t,x)=J^\pi(t,x)$ for every $(t,x)\in[0,T]\times\mathbb R^d$ and $q(t,x,a)=q^\pi(t,x,a)$ $dt\otimes dx\otimes da$-almost everywhere;

    \item for almost every initial condition $(t,x)$ and every sampling grid $G$, the process $M^{\mu,G}[u,q]$ in \eqref{equation:observable-process} is an $(\mathbb F^G,\mathbb P)$-martingale.
\end{enumerate}
Moreover, if these equivalent conditions hold, then, for every behavior policy $\nu\in\Pi$, every initial condition $(t,x)$, and every sampling grid $G$, the process
$M^{\nu,G}[u,q]$ is an $(\mathbb F^G,\mathbb P)$-martingale.
\end{theorem}

\begin{proof}
We prove the two implications separately.

\medskip\noindent
\emph{\textnormal{(i)} $\Rightarrow$ \textnormal{(ii)}.}
Suppose that $u(t,x)=J^\pi(t,x)$ for all $(t,x)\in[0,T]\times\mathbb R^d$ and $q(t,x,a)=q^\pi(t,x,a)$ $dt\otimes dx\otimes da$-almost everywhere. Then the martingale property follows directly from Theorem~\ref{theorem:q-martingale-characterization}. The same theorem shows that it holds for every admissible behavior policy, every initial condition, and every sampling grid.

\medskip\noindent
\emph{\textnormal{(ii)} $\Rightarrow$ \textnormal{(i)}.}
Fix an initial condition $(t,x)$ and a sampling grid $G$ for which the assumptions of the theorem hold, and write $\mathbb Q=\mathbb Q^{\mu,G}$, $B=B^{\mu,G}$, $Y_s=u(s,X_s)$. We first reconstruct a semimartingale decomposition of $Y$ from the martingale property of $M^{\mu,G}[u,q]$. On each interval
$[t_i,t_{i+1})$, no new randomization variable is revealed after time $t_i$, and the filtration is generated by $\mathcal F_{t_i}^G$ together with the subsequent Brownian increments. Since these increments are independent of $\mathcal F_{t_i}^G$ and $M^{\mu,G}[u,q]$ is an $(\mathbb F^G,\mathbb P)$-martingale, the Brownian martingale representation theorem, yields an $\mathbb R^n$-valued, $\mathbb F^G$-predictable process $H^i$ such that
\[
    M_s^{\mu,G}[u,q]-M_{t_i}^{\mu,G}[u,q]
    =
    \int_{t_i}^s
    \langle H_v^i,dW_v\rangle,
    \qquad s\in[t_i,t_{i+1}),
\]
with
\[
    \int_{t_i}^{t_{i+1}}|H_v^i|^2\,dv<+\infty,
    \qquad \mathbb P\text{-a.s.}
\]
Setting $Z_v^i\coloneqq e^{\beta(v-t)}H_v^i$, we equivalently obtain
\begin{equation}\label{equation:representation-M}
    M_s^{\mu,G}[u,q]-M_{t_i}^{\mu,G}[u,q]
    =
    \int_{t_i}^s
    e^{-\beta(v-t)}
    \langle Z_v^i,dW_v\rangle.
\end{equation}
At each intermediate sampling time, the filtration is enlarged by a new randomization variable. Since $M^{\mu,G}[u,q]$ is continuous on $[t,T)$, no martingale jump occurs at these times, and the
interval-wise representations can be concatenated. Moreover, since $\mathcal F^G_{T^-}=\mathcal F^G_T$ by construction,
\[
    M_{T^-}^{\mu,G}[u,q] = \mathbb E\left[ M_{T}^{\mu,G}[u,q]\,\big|\, \mathcal F^G_{T^-}\right] =  M_{T}^{\mu,G}[u,q],\qquad \mathbb P\text{-a.s.}
\]
Hence the martingale $M^{\mu,G}[u,q]$ is continuous on $[t,T]$ and the representation \eqref{equation:representation-M} holds on the whole last interval $[t_{N_G-1},T]$. Defining
\[
    Z_s
    \coloneqq
    \sum_{i=0}^{N_G-1}
    Z_s^i\mathbbm 1_{(t_i,t_{i+1}]}(s),
\]
we then obtain an $\mathbb F^G$-predictable process satisfying
\[
    \int_t^T|Z_s|^2\,ds<+\infty,
    \qquad \mathbb P\text{-a.s.},
\]
and
\[
    M_s^{\mu,G}[u,q]
    =
    u(t,x)
    +
    \int_t^s
    e^{-\beta(v-t)}
    \langle Z_v,dW_v\rangle,
    \qquad s\in[t,T].
\]
The process $Z$ is unique up to $ds\otimes d\mathbb P$-almost everywhere equality.

\noindent Using the definition of $M^{\mu,G}[u,q]$ and
\[
    dW_s
    =
    dB_s
    -
    \psi(s,X_s,\alpha_s)\,ds,
\]
we obtain
\begin{equation}
\label{equation:joint-reconstructed-semimartingale}
    dY_s
    =
    -S_s\,ds
    +
    \langle Z_s,dB_s\rangle,
\end{equation}
where
\begin{equation}
\label{equation:joint-reconstructed-drift}
    S_s
    =
    r(s,X_s,\alpha_s)
    -
    q(s,X_s,\alpha_s)
    +
    \langle Z_s,\psi(s,X_s,\alpha_s)\rangle
    -
    \beta Y_s.
\end{equation}
By \eqref{equation:candidate-integrability}, the polynomial growth
of $u$ and $r$, the moment estimates for $X$ and $\alpha$, the
boundedness of $\psi$, and the local square integrability of $Z$,
\[
    \int_t^T |S_s|\,ds<+\infty,
    \qquad \mathbb Q\text{-a.s.}
\]
Thus \eqref{equation:joint-reconstructed-semimartingale} is an
$\mathbb F^G$-semimartingale decomposition of $Y$ under
$\mathbb Q$. On the other hand, since $X$ is adapted to $\mathbb F^B$, so is
$Y_s=u(s,X_s)$. Since
$\mathbb F^B\subseteq\mathbb F^G$, Stricker's theorem \cite{Stricker1977} (see also \cite{BieleckiJakubowskiJeanblancNieweglowski2021}) implies that
$Y$ is also an $\mathbb F^B$-semimartingale.
Its local martingale part admits a representation with respect to $B$, since $\mathbb F^B$ is a Brownian filtration. Comparing this decomposition with
\eqref{equation:joint-reconstructed-semimartingale} in the larger filtration $\mathbb F^G$, and using uniqueness of the continuous
semimartingale decomposition, shows that $S$ and $Z$ admit $\mathbb F^B$-progressively measurable versions. We henceforth work with these versions.

\noindent Rearranging \eqref{equation:joint-reconstructed-drift}, we have
\[
    q(s,X_s,\alpha_s)
    =
    r(s,X_s,\alpha_s)
    +
    \langle Z_s,\psi(s,X_s,\alpha_s)\rangle
    -
    S_s
    -
    \beta Y_s,
    \qquad
    ds\otimes d\mathbb Q\text{-a.e.}
\]
Since $S$, $Z$, $Y$, and $X$ are $\mathbb F^B$-adapted,
the conditional law of $\alpha_s$ given $\mathcal F_s^B$,
together with the strict positivity of the behavior-policy
density, yields
\begin{equation}
\label{equation:joint-reconstructed-q}
    q(s,X_s,a)
    =
    r(s,X_s,a)
    +
    \langle Z_s,\psi(s,X_s,a)\rangle
    -
    S_s
    -
    \beta Y_s,\qquad\text{$ds\otimes d\mathbb Q\otimes da$-a.e.}
\end{equation}
By Proposition~\ref{proposition:densities}, the normalization \eqref{equation:joint-normalization} can be
evaluated along the reference state $X$. Integrating \eqref{equation:joint-reconstructed-q} with respect to
$\pi(a\,|\,s,X_s)\,da$ and using
\eqref{equation:joint-normalization}, we obtain
\[
    S_s
    =
    \langle Z_s,\bar\psi_\pi(s,X_s)\rangle
    +
    R_\pi(s,X_s)
    -
    \beta Y_s
    =
    f_\pi(s,X_s,Y_s,Z_s),
    \qquad
    ds\otimes d\mathbb Q\text{-a.e.}
\]
Together with
\eqref{equation:joint-reconstructed-semimartingale} and the
terminal condition $Y_T=u(T,X_T)=g(X_T)$, this shows that $(Y,Z)$ satisfies the policy BSDE
\eqref{equation:policy-bsde}.

\noindent To apply the uniqueness result for
\eqref{equation:policy-bsde}, it suffices to verify that the reconstructed pair $(Y,Z)$ belongs to $\mathcal S^p_{t,T}(\mathbb R;\mathbb Q,\mathbb F^B)\times
\mathcal H^p_{t,T}(\mathbb R^{n};\mathbb Q,\mathbb F^B)$ for some $p>1$. We verify this with $p=2$.

\noindent Since $u$ has polynomial growth and the reference diffusion has moments of every order, $Y\in
    \mathcal S^2_{t,T}
    (\mathbb R;\mathbb Q,\mathbb F^B)$. Moreover, $f_\pi(\cdot,X_\cdot,0,0)$ is square integrable by the polynomial growth of $R_\pi$. Since $f_\pi$ is uniformly Lipschitz in $(y,z)$ and $\int_t^T|Z_s|^2\,ds<+\infty$ almost surely, \cite[Lemma~3.1]{BriandDelyonHuPardouxStoica2003} yields $Z\in
    \mathcal H^2_{t,T}
    (\mathbb R^{n};\mathbb Q,\mathbb F^B)$. Hence
\begin{align*}
    Y_s
    =
    Y_s^{t,x,\pi},
    \quad
    s\in[t,T],
    \quad
    \mathbb Q\text{-a.s.},\qquad\qquad Z_s = Z^{t,x,\pi}_s,\quad ds\otimes d\mathbb Q\text{-a.e.} 
\end{align*}
Proposition~\ref{proposition:bsde-value} then gives
\[
    u(s,X_s)
    =
    J^\pi(s,X_s),
    \qquad
    s\in[t,T],
    \quad
    \mathbb Q\text{-a.s.}
\]
In particular, since $X_t=x$, $u(t,x)=J^\pi(t,x)$ for every initial condition in the full-measure set for which the martingale property and
\eqref{equation:candidate-integrability} hold. Hence
\[
    u=J^\pi
    \qquad
    dt\otimes dx\text{-a.e.}
\]
Both $u$ and $J^\pi$ are continuous on $[0,T)\times\mathbb R^d$ and they coincide at time $T$ by \eqref{equation:joint-terminal-condition}. Hence, we conclude that
\begin{equation}
\label{equation:joint-value-identification}
    u(t,x)
    =
    J^\pi(t,x),
    \qquad
    (t,x)\in[0,T]\times\mathbb R^d.
\end{equation}
It remains to identify $q$. By
\eqref{equation:joint-value-identification}, condition
\textnormal{(ii)} becomes
\[
    M^{\mu,G}[J^\pi,q]
    \quad\text{is an }
    (\mathbb F^G,\mathbb P)\text{-martingale}
\]
for almost every initial condition and every sampling grid.
Theorem~\ref{theorem:q-martingale-characterization} therefore
yields
\[
    q=q^\pi
    \qquad
    dt\otimes dx\otimes da\text{-a.e.},
\]
which concludes the proof.
\end{proof}

\begin{remark}
\label{remark:joint-characterization-on-off-policy}
The observations of Remark~\ref{remark:q-characterization-on-off-policy} on the behavior policy apply to Theorem~\ref{theorem:joint-martingale-characterization} as well. As in \cite[Theorem~7]{JiaZhou2025Erratum}, the target policy is specified through the
normalization \eqref{equation:joint-normalization}.

\medskip\noindent 
The additional optimality statement in
\cite[Theorem~7]{JiaZhou2025Erratum} is already contained in our policy improvement result (Theorem~\ref{theorem:policy-improvement}). Indeed, if, in addition, $\pi$ coincides with the Gibbs policy generated by the candidate $q$, then the preceding
characterization gives $q=q^\pi$ and $u=J^\pi$. Hence $\pi$ is a fixed point of the Gibbs update, and
Theorem~\ref{theorem:policy-improvement} yields $J^\pi=J^*$. Thus $\pi$ is optimal within $\Pi$ and $u=J^*$.
\end{remark}

\section{Optimal value and \texorpdfstring{$q$}{q}-function} \label{section:optimal-value-and-q}

We now turn from the evaluation of a fixed target policy to the direct characterization of the optimal value
\[
    J^*(t,x)=\sup_{\pi\in\Pi}J^\pi(t,x).
\]
The estimates obtained for individual policy values depend in general on the policy and therefore do not provide, by themselves, a uniform bound on $J^*$. We instead construct a BSDE whose driver is obtained by maximizing, pointwise, the entropy-regularized policy drivers over probability densities on the
action space. 

\subsection{The optimal BSDE}

For $(t,x,y,z)\in[0,T]\times\mathbb R^d\times\mathbb R\times\mathbb R^{n}$, define
\begin{equation}
\label{equation:optimal-driver}
    f_*(t,x,y,z)\coloneqq\gamma\log\left(
    \int_A\exp\left(\frac{\langle z,\psi(t,x,a)\rangle+r(t,x,a)}{\gamma}
    \right)\,da\right)-\beta y.
\end{equation}

\noindent To ensure that this driver is finite and satisfies the growth conditions required by the standard BSDE theory, we impose an additional assumption on the running reward $r$ throughout the section.

\begin{assumption}
\label{assumption:separable-reward}
There exist a Borel measurable function $\widetilde r:[0,T]\times\mathbb R^d\to\mathbb R$, a Borel measurable function $\chi:A\to\mathbb R$, and constants $C>0$ and $m\geq1$ such that, for every $(t,x,a)\in[0,T]\times\mathbb R^d\times A$,
\begin{align}
    &r(t,x,a)\leq\widetilde r(t,x)+\chi(a),
    \label{equation:separable-reward}\\
    &|\widetilde r(t,x)|\leq C(1+|x|^m),
    \label{equation:separable-reward-growth}\\
    &\int_A\exp\left(\frac{\chi(a)}{\gamma}\right)\,da
    <+\infty.
    \label{equation:action-reward-integrability}
\end{align}
\end{assumption}

\noindent Since $A$ has positive Lebesgue measure and $\chi$ is finite, the integral in \eqref{equation:action-reward-integrability} is strictly positive.

\begin{remark}\label{remark:bounded-action-space}
If the action space $A$ is bounded, Assumption~\ref{assumption:separable-reward} is automatically satisfied, and therefore imposes no restriction beyond Assumption~\ref{assumption:coefficients}. Indeed, let $R>0$ be such that $|a|\leq R$ for every $a\in A$. Then \eqref{equation:reward-growth} yields
\[
    r(t,x,a)\leq C\bigl(1+|x|^m+R^m\bigr),\qquad (t,x,a)\in[0,T]\times\mathbb R^d\times A,
\]
so that \eqref{equation:separable-reward}--\eqref{equation:action-reward-integrability} hold with $\widetilde r(t,x)\coloneqq C(1+|x|^m+R^m)$ and $\chi\equiv0$ $($with $C$ replaced by the larger constant $C(1+R^m)$ in \eqref{equation:separable-reward-growth}$)$, the integral in \eqref{equation:action-reward-integrability} being the Lebesgue measure of $A$, which is finite.

\medskip\noindent 
When $A$ is unbounded, by contrast, the standing growth condition \eqref{equation:reward-growth} alone does not ensure the finiteness of the action integral in \eqref{equation:optimal-driver}. For instance, if $A=\mathbb R$ and $r(t,x,a)=|a|$, then $r$ satisfies \eqref{equation:reward-growth}, but
\[
    \int_{\mathbb R}
    \exp\left(\frac{r(t,x,a)}{\gamma}\right)\,da
    =
    \int_{\mathbb R}
    \exp\left(\frac{|a|}{\gamma}\right)\,da
    =
    +\infty.
\]
\end{remark}

\medskip\noindent
The next result states the properties of $f_*$ that will be used below.

\begin{proposition}
\label{proposition:optimal-driver}
The driver $f_*$ is finite and Borel measurable. Moreover, there exist constants $C>0$ and $m\geq1$ such that
\begin{align}
    |f_*(t,x,y,z)-f_*(t,x,y',z')|
    &\leq
    \beta|y-y'|+\|\psi\|_\infty|z-z'|,
    \label{equation:optimal-driver-lipschitz}\\
    |f_*(t,x,0,0)|
    &\leq
    C(1+|x|^m),
    \label{equation:optimal-driver-growth}
\end{align}
for every $(t,x)\in[0,T]\times\mathbb R^d$,
$y,y'\in\mathbb R$, and $z,z'\in\mathbb R^{n}$.

\noindent In addition, for every $\nu\in\Pi$,
\begin{equation}
\label{equation:driver-domination}
    f_\nu(t,x,y,z)\leq f_*(t,x,y,z),\qquad \text{for every $(t,x,y,z)$.}
\end{equation}
\end{proposition}

\begin{proof}
By \eqref{equation:separable-reward} and boundedness of $\psi$,
\[
    \int_A
    \exp\left(
        \frac{
            \langle z,\psi(t,x,a)\rangle+r(t,x,a)
        }{\gamma}
    \right)\,da\leq
    \exp\left(
        \frac{
            \|\psi\|_\infty|z|+\widetilde r(t,x)
        }{\gamma}
    \right)
    \int_A
    \exp\left(\frac{\chi(a)}{\gamma}\right)\,da,
\]
where, by \eqref{equation:action-reward-integrability}, 
\[
    0<
    \int_A
    \exp\left(\frac{\chi(a)}{\gamma}\right)\,da
    <+\infty.
\]
On the other hand, \eqref{equation:reward-growth} implies $r(t,x,a)
    \geq
    -C(1+|x|^m+|a|^m)$, and therefore
\begin{align*}
  \int_A
    \exp\left(
        \frac{
            \langle z,\psi(t,x,a)\rangle+r(t,x,a)
        }{\gamma}
    \right)\,da\geq
    \exp\left(
        -\frac{
            \|\psi\|_\infty|z|
            +C(1+|x|^m)
        }{\gamma}
    \right)
    \int_A
    \exp\left(
        -\frac{C|a|^m}{\gamma}
    \right)\,da.
\end{align*}
Since $A$ has positive Lebesgue measure and $a\longmapsto\exp\left(-\frac{C|a|^m}{\gamma}\right)$ is integrable on $\mathbb R^k$, we have
\[
    0<
    \int_A
    \exp\left(-\frac{C|a|^m}{\gamma}\right)\,da
    <+\infty.
\]
Hence the action integral in
\eqref{equation:optimal-driver} is finite and strictly positive. Taking logarithms gives
\begin{align}
    -\|\psi\|_\infty|z|
    -C(1+|x|^m)
    +
    &\gamma\log
    \int_A
    \exp\left(
        -\frac{C|a|^m}{\gamma}
    \right)\,da\notag\\
    &\leq
    f_*(t,x,0,z)\leq
    \|\psi\|_\infty|z|
    +\widetilde r(t,x)
    +
    \gamma\log
    \int_A
    \exp\left(\frac{\chi(a)}{\gamma}\right)\,da.
    \label{equation:optimal-driver-two-sided-bound}
\end{align}
Next, boundedness of $\psi$ gives $|\langle z-z',\psi(t,x,a)\rangle|
    \leq
    \|\psi\|_\infty|z-z'|$. Hence
\begin{align*}
    &\exp\left(
        -\frac{\|\psi\|_\infty|z-z'|}{\gamma}
    \right)
    \int_A
    \exp\left(
        \frac{
            \langle z',\psi(t,x,a)\rangle+r(t,x,a)
        }{\gamma}
    \right)\,da
    \\
    &\qquad\leq
    \int_A
    \exp\left(
        \frac{
            \langle z,\psi(t,x,a)\rangle+r(t,x,a)
        }{\gamma}
    \right)\,da
    \\
    &\qquad\leq
    \exp\left(
        \frac{\|\psi\|_\infty|z-z'|}{\gamma}
    \right)
    \int_A
    \exp\left(
        \frac{
            \langle z',\psi(t,x,a)\rangle+r(t,x,a)
        }{\gamma}
    \right)\,da.
\end{align*}
Taking logarithms and using the affine dependence of $f_*$ on $y$
gives
\[
    |f_*(t,x,y,z)-f_*(t,x,y',z')|
    \leq
    \beta|y-y'|
    +
    \|\psi\|_\infty|z-z'|,
\]
which proves \eqref{equation:optimal-driver-lipschitz}. Taking $z=0$ in
\eqref{equation:optimal-driver-two-sided-bound} and using \eqref{equation:separable-reward-growth} yields \eqref{equation:optimal-driver-growth}, possibly after enlarging
$C$ and $m$. Finally, let $\nu\in\Pi$. Since $\nu(a\,|\, t,x)>0$, Jensen's
inequality gives
\begin{align*}
    &\log
    \int_A
    \exp\left(
        \frac{
            \langle z,\psi(t,x,a)\rangle+r(t,x,a)
        }{\gamma}
    \right)\,da
    =
    \log
    \int_A
    \frac{
        \exp\left(
            \frac{
                \langle z,\psi(t,x,a)\rangle+r(t,x,a)
            }{\gamma}
        \right)
    }{
        \nu(a\,|\, t,x)
    }
    \nu(a\,|\, t,x)\,da
    \\
    &\geq
    \int_A
    \left[
        \frac{
            \langle z,\psi(t,x,a)\rangle+r(t,x,a)
        }{\gamma}
        -\log\nu(a\,|\, t,x)
    \right]
    \nu(a\,|\, t,x)\,da.
\end{align*}
Multiplying by $\gamma$ and subtracting $\beta y$ yields
\[
    f_\nu(t,x,y,z)
    \leq
    f_*(t,x,y,z),
\]
which proves \eqref{equation:driver-domination}.
\end{proof}

\noindent We now consider the BSDE associated with the driver $f_*$. On an arbitrary Brownian probability space $(\mathbb Q,B,\mathbb F^B)$ supporting the reference diffusion $X^{t,x}$, let
\begin{equation}\label{equation:optimal-bsde}
    Y_s^{t,x,*}=g(X_T^{t,x})+\int_s^T f_*
    \bigl(v,X_v^{t,x},Y_v^{t,x,*},Z_v^{t,x,*}
    \bigr)\,dv-\int_s^T\langle
    Z_v^{t,x,*},dB_v\rangle,\qquad s\in[t,T].
\end{equation}

\begin{proposition}
\label{proposition:optimal-bsde}
For every $(t,x)\in[0,T]\times\mathbb R^d$ and every $p>1$, equation~\eqref{equation:optimal-bsde} admits a unique solution $(Y^{t,x,*},Z^{t,x,*})\in\mathcal S^p_{t,T}
    (\mathbb R;\mathbb Q,\mathbb F^B)\times\mathcal H^p_{t,T}(\mathbb R^{n};\mathbb Q,\mathbb F^B)$. Moreover, there exist constants $C_p>0$ and $m_p\geq1$ such that
\[
    \mathbb E^{\mathbb Q}\left[\sup_{t\leq s\leq T}|Y_s^{t,x,*}|^p +\left(\int_t^T|Z_s^{t,x,*}|^2\,ds\right)^{p/2}
    \right]\leq C_p(1+|x|^{m_p}).
\]
There exist a Borel measurable function $u_*:[0,T]\times\mathbb R^d\to\mathbb R$ of polynomial growth, continuous on $[0,T)\times\mathbb R^d$, and a Borel function $z_*:[0,T]\times\mathbb R^d\to\mathbb R^{n}$ such that, for every initial condition,
\begin{align}
    Y_s^{t,x,*}&= u_*(s,X_s^{t,x}),\qquad\mathbb Q\text{-a.s.},\;
    s\in[t,T],
    \label{equation:optimal-value-representation}\\
    Z_s^{t,x,*}&= z_*(s,X_s^{t,x}),\qquad
    ds\otimes d\mathbb Q\text{-a.e.}
    \label{equation:optimal-integrand-representation}
\end{align}
In particular, $u_*(T,x)=g(x)$. The function $z_*$ may be chosen independently of the Brownian probability space and is unique up to
$dt\otimes dx$-almost everywhere equality.
If $g$ is continuous, then $u_*$ is continuous on $[0,T]\times\mathbb R^d$.
\end{proposition}

\begin{proof}
By Proposition~\ref{proposition:optimal-driver}, the driver $f_*$ is uniformly Lipschitz in $(y,z)$ and $f_*(\cdot,X^{t,x}_\cdot,0,0)$ has moments of every order, while $g(X_T^{t,x})\in L^p(\mathbb Q)$ for every $p>1$. Hence the same argument as in Proposition~\ref{proposition:bsde-value}, based on \cite[Proposition~3.2 and Theorem~4.2]{BriandDelyonHuPardouxStoica2003}, yields existence, uniqueness, and the stated $L^p$ estimates for $(Y^{t,x,*},Z^{t,x,*})$. Applying \cite[Theorem~4.1]{ElKarouiPengQuenez1997} to the Markovian BSDE \eqref{equation:optimal-bsde}, exactly as in the proof of Proposition~\ref{proposition:markovian-integrand}, yields a Borel function $z_*$ such that \eqref{equation:optimal-integrand-representation} holds. The arguments used in that proof also show that $z_*$ may be chosen independently of the Brownian probability space and is unique up to $dt\otimes dx$-almost everywhere equality.

\noindent Set $h_*(t,x)\coloneqq f_*(t,x,0,0)$ and
\[
    \zeta_*(t,x)\coloneqq
    \begin{cases}
    \displaystyle
    \frac{f_*(t,x,0,z_*(t,x))-f_*(t,x,0,0)}
         {|z_*(t,x)|^2}\,z_*(t,x),
        & z_*(t,x)\neq0,\\[1ex]
    0, & z_*(t,x)=0.
    \end{cases}
\]
By Proposition~\ref{proposition:optimal-driver}, $h_*$
is Borel measurable with polynomial growth, while $\zeta_*$ is Borel measurable and satisfies
$|\zeta_*(t,x)|\le\|\psi\|_\infty$. Moreover,
\[
    f_*(t,x,y,z_*(t,x))
    =
    \langle z_*(t,x),\zeta_*(t,x)\rangle
    +h_*(t,x)-\beta y.
\]
Consequently, by \eqref{equation:optimal-integrand-representation}, $(Y^{t,x,*},Z^{t,x,*})$ solves the linear BSDE \eqref{equation:linear-bsde} with
$(\zeta,h)=(\zeta_*,h_*)$ for every initial condition.
Proposition~\ref{proposition:continuity} therefore shows that
$u_*(t,x)\coloneqq Y_t^{t,x,*}$ is deterministic,
Borel measurable, and of polynomial growth, satisfies
\eqref{equation:optimal-value-representation} and
$u_*(T,\cdot)=g$, and is continuous on
$[0,T)\times\mathbb R^d$.
If $g$ is continuous, then $u_*$ is continuous on
$[0,T]\times\mathbb R^d$.
\end{proof}

\subsection{The optimal \texorpdfstring{$q$}{q}-function and verification}

Fix a Borel representative of $z_*$. We define
\begin{equation}
\label{equation:optimal-q-definition}
    q_*(t,x,a)\coloneqq\langle z_*(t,x),\psi(t,x,a)\rangle
    +
    r(t,x,a)
    -
    f_*(t,x,0,z_*(t,x)).
\end{equation}
The function $q_*$ is Borel measurable. A different Borel representative of $z_*$ gives the same $q_*$ outside a common $dt\otimes dx$-null set.

\medskip\noindent The definition immediately yields the exponential normalization
\begin{equation}\label{equation:optimal-q-normalization}
    \int_A\exp\left(\frac{q_*(t,x,a)}{\gamma}
    \right)\,da=1,\qquad (t,x)\in[0,T]\times\mathbb R^d.
\end{equation}
Indeed, by \eqref{equation:optimal-q-definition} and
\eqref{equation:optimal-driver},
\[
\begin{aligned}
    \exp\left(
        \frac{q_*(t,x,a)}{\gamma}
    \right)
    &=
    \frac{
        \displaystyle
        \exp\left(
            \frac{
                \langle z_*(t,x),\psi(t,x,a)\rangle
                +r(t,x,a)
            }{\gamma}
        \right)
    }{
        \displaystyle
        \int_A
        \exp\left(
            \frac{
                \langle z_*(t,x),\psi(t,x,a')\rangle
                +r(t,x,a')
            }{\gamma}
        \right)\,da'
    },
\end{aligned}
\]
which, integrated with respect to $a$, gives
\eqref{equation:optimal-q-normalization}.

\medskip\noindent Accordingly, define
\begin{equation}\label{equation:optimal-policy}
    \pi^*(a\,|\, t,x)
    \coloneqq
    \exp\left(
        \frac{q_*(t,x,a)}{\gamma}
    \right).
\end{equation}

\begin{remark}
\label{remark:optimal-gibbs-policy}
The definition of $\pi^*$ is consistent with the Gibbs update
\eqref{equation:gibbs-update} introduced in
Section~\ref{section:policy-improvement}. In particular, since the
optimal $q$-function is already exponentially normalized by
\eqref{equation:optimal-q-normalization}, $\pi^*$ is precisely the
Gibbs density generated by $q_*$. Moreover, the additional
assumptions of the present section imply the integrability condition
\eqref{equation:gibbs-integrability} used in
Section~\ref{section:policy-improvement}. However,
Assumption~\ref{assumption:separable-reward} alone does not ensure
that $\pi^*$ satisfies all the admissibility requirements of
Definition~\ref{definition:policy}. We therefore explicitly assume
that $\pi^*$ admits a representative in $\Pi$, still denoted by $\pi^*$, whenever admissibility is needed.
\end{remark}

\begin{proposition}
\label{proposition:optimal-verification}
For every $(t,x)\in[0,T]\times\mathbb R^d$,
\begin{equation}\label{equation:optimal-value-bound}
    -\infty< J^*(t,x)\leq u_*(t,x)<+\infty.
\end{equation}
Suppose, in addition, that the density in
\eqref{equation:optimal-policy} admits a representative in $\Pi$, still denoted by $\pi^*$, agreeing with
\eqref{equation:optimal-policy} $dt\otimes dx\otimes da$-almost everywhere. Then $\pi^*$ is
optimal within $\Pi$, and
\[
    u_*=J^{\pi^*}=J^*\qquad\text{everywhere},\qquad q_*=q^{\pi^*}
    \qquad
    dt\otimes dx\otimes da\text{-a.e.}
\]
\end{proposition}

\begin{proof}
Fix $\nu\in\Pi$. By
\eqref{equation:driver-domination},
\[
    f_\nu(t,x,y,z)
    \leq
    f_*(t,x,y,z)
\]
for all $(t,x,y,z)$. The BSDEs associated with $\nu$ and with $f_*$ have the same terminal condition $g(X_T^{t,x})$. The comparison theorem for scalar BSDEs (see \cite[Theorem~6.2.2]{Pham2009}) therefore gives
\[
    J^\nu(t,x)
    =
    Y_t^{t,x,\nu}
    \leq
    Y_t^{t,x,*}
    =
    u_*(t,x).
\]
Taking the supremum over $\nu\in\Pi$ yields
\[
    J^*(t,x)\leq u_*(t,x).
\]
Since $\Pi$ is nonempty and every policy value is finite, $J^*(t,x)>-\infty$, while finiteness of $u_*$ follows from Proposition~\ref{proposition:optimal-bsde}. This proves \eqref{equation:optimal-value-bound}.

\noindent Suppose now that $\pi^*\in\Pi$. By
\eqref{equation:optimal-policy},
\[
    \gamma\log\pi^*(a\,|\, t,x)
    =
    q_*(t,x,a)
\]
for $dt\otimes dx\otimes da$-almost every $(t,x,a)$.
Using \eqref{equation:optimal-q-definition}, we then obtain, for
every $y\in\mathbb R$ and for $dt\otimes dx$-almost every $(t,x)$,
\begin{align*}
    f_{\pi^*}(t,x,y,z_*(t,x))
    &=
    \int_A
    \Big[
        \langle z_*(t,x),\psi(t,x,a)\rangle
        +r(t,x,a)
        -\gamma\log\pi^*(a\,|\, t,x)
    \Big]
    \pi^*(a\,|\, t,x)\,da
    -\beta y
    \\
    &=
    f_*(t,x,0,z_*(t,x))
    \int_A\pi^*(a\,|\, t,x)\,da
    -\beta y
    =
    f_*(t,x,y,z_*(t,x)).
\end{align*}
Hence
\begin{equation}
\label{equation:optimal-driver-attainment}
    f_{\pi^*}(t,x,y,z_*(t,x))
    =
    f_*(t,x,y,z_*(t,x)).
\end{equation}
By Proposition~\ref{proposition:densities},
\eqref{equation:optimal-driver-attainment} remains valid when evaluated along the reference diffusion. Hence $(Y^{t,x,*},Z^{t,x,*})$ solves the policy BSDE associated with $\pi^*$. Uniqueness gives
\[
    u_*(t,x)=J^{\pi^*}(t,x),\qquad\text{for every }(t,x).
\]
Therefore,
\[
    J^*(t,x)
    \leq
    u_*(t,x)
    =
    J^{\pi^*}(t,x)
    \leq
    J^*(t,x),
\]
and consequently $u_*=J^{\pi^*}=J^*$. Thus $\pi^*$ is optimal within $\Pi$.

\noindent Uniqueness of the martingale integrand in the two BSDEs gives
\[
    z_*=z_{\pi^*}
    \qquad
    dt\otimes dx\text{-a.e.}
\]
Substituting this identity and $\gamma\log\pi^*=q_*$ into \eqref{equation:q-definition}, or equivalently comparing \eqref{equation:optimal-q-definition} with the policy $q$-function, yields
\[
    q_*=q^{\pi^*}
    \qquad
    dt\otimes dx\otimes da\text{-a.e.}
\]
\end{proof}

\subsection{Martingale characterization}

We finally characterize the pair $(u_*,q_*)$ without prescribing a target policy. The policy-dependent normalization used in Section~\ref{section:policy-evaluation} is replaced here by the exponential normalization \eqref{equation:optimal-q-normalization}. The result corresponds to \cite[Theorem~9]{JiaZhou2025Erratum}. As in Theorem~\ref{theorem:joint-martingale-characterization}, the candidate $u$ is only required to be continuous on $[0,T)\times\mathbb R^d$, and $u_*$ has this property by Proposition~\ref{proposition:optimal-bsde}.

\begin{theorem}
\label{theorem:optimal-martingale-characterization}
Fix a behavior policy $\mu\in\Pi$. Let $u:[0,T]\times\mathbb R^d\to\mathbb R$ be Borel measurable with polynomial growth and continuous on $[0,T)\times\mathbb R^d$, and let $q:[0,T]\times\mathbb R^d\times A\to\mathbb R$ be Borel measurable and suppose that
\eqref{equation:candidate-integrability} holds for almost every initial condition $(t,x)$ and every sampling grid $G$. Suppose that
\begin{align}
    u(T,x)=g(x),&\qquad x\in\mathbb R^d,\label{equation:optimal-candidate-terminal}\\
    \int_A
    \exp\left(
        \frac{q(t,x,a)}{\gamma}
    \right)\,da
    =
    1,&\qquad\text{for $dt\otimes dx$-a.e.  $(t,x)$.}\label{equation:optimal-candidate-normalization}
\end{align}
Then the following statements are equivalent:
\begin{enumerate}[(i)]
    \item $u=u_*$ everywhere and $q=q_*$ $dt\otimes dx\otimes da$-a.e.;
    \item for almost every initial condition $(t,x)$ and every sampling grid $G$, the process $M^{\mu,G}[u,q]$ is an $(\mathbb F^G,\mathbb P)$-martingale.
\end{enumerate}
Moreover, if these equivalent conditions hold, then
$M^{\nu,G}[u,q]$ is an $(\mathbb F^G,\mathbb P)$-martingale for every behavior policy $\nu\in\Pi$, every initial condition, and every sampling grid.

\noindent If, in addition, $\exp\left(\frac{q(t,x,a)}{\gamma}
\right)$ admits an admissible representative in $\Pi$, then this representative is an optimal policy and
\[
    u=J^*.
\]
\end{theorem}

\begin{proof}
\emph{\textnormal{(i)} $\Rightarrow$ \textnormal{(ii)}.}
Fix $(t,x,\mu,G)$. Realizing the optimal BSDE
\eqref{equation:optimal-bsde} on
$(\mathbb Q^{\mu,G},B^{\mu,G})$ and using
\[
    dB_s^{\mu,G}
    =
    dW_s+\psi(s,X_s,\alpha_s)\,ds,
\]
we obtain, exactly as in the proof of
Theorem~\ref{theorem:q-martingale-characterization},
\begin{align}
    M_s^{\mu,G}[u_*,q]
    &=
    u_*(t,x)
    +
    \int_t^s
    e^{-\beta(v-t)}
    \langle z_*(v,X_v),dW_v\rangle
    \notag\\
    &\quad+
    \int_t^s
    e^{-\beta(v-t)}
    \bigl[
        q_*(v,X_v,\alpha_v)
        -q(v,X_v,\alpha_v)
    \bigr]\,dv,
    \label{equation:optimal-martingale-decomposition}
\end{align}
where the stochastic integral is a square-integrable
$(\mathbb F^G,\mathbb P)$-martingale.

\noindent If $q=q_*$ $dt\otimes dx\otimes da$-almost everywhere, Proposition~\ref{proposition:densities} gives
\[
    q(s,X_s,\alpha_s)
    =
    q_*(s,X_s,\alpha_s)
    \qquad
    ds\otimes d\mathbb P\text{-a.e.}
\]
Hence the finite-variation term in
\eqref{equation:optimal-martingale-decomposition} vanishes and $M^{\mu,G}[u_*,q]$ is a martingale. The same argument applies to every admissible behavior policy $\nu\in\Pi$, proving the additional
martingality statement.

\medskip\noindent\emph{\textnormal{(ii)} $\Rightarrow$ \textnormal{(i)}.}
Fix an initial condition $(t,x)$ and a sampling grid $G$ for which the assumptions hold, and set
\[
    \mathbb Q\coloneqq\mathbb Q^{\mu,G},
    \qquad
    B\coloneqq B^{\mu,G},
    \qquad
    Y_s\coloneqq u(s,X_s).
\]
Arguing as in the proof of Theorem~\ref{theorem:joint-martingale-characterization}, the martingale property of $M^{\mu,G}[u,q]$ yields $\mathbb F^B$-progressively measurable processes $S$ and $Z$ such that
\[
    dY_s=-S_s\,ds+\langle Z_s,dB_s\rangle
\]
and
\begin{equation}
\label{equation:optimal-reconstruction-q}
    q(s,X_s,a)
    =
    r(s,X_s,a)
    +
    \langle Z_s,\psi(s,X_s,a)\rangle
    -
    S_s
    -
    \beta Y_s,\qquad\text{$ds\otimes d\mathbb Q\otimes da$-a.e.}
\end{equation}
By Proposition~\ref{proposition:densities}, \eqref{equation:optimal-candidate-normalization} can be evaluated
along $X$. Exponentiating
\eqref{equation:optimal-reconstruction-q}, integrating over $A$, and using \eqref{equation:optimal-candidate-normalization} gives
\[
    1
    =
    \exp\left(
        -\frac{S_s+\beta Y_s}{\gamma}
    \right)
    \int_A
    \exp\left(
        \frac{
            \langle Z_s,\psi(s,X_s,a)\rangle
            +r(s,X_s,a)
        }{\gamma}
    \right)\,da.
\]
Hence, by \eqref{equation:optimal-driver},
\[
    S_s=f_*(s,X_s,Y_s,Z_s),
    \qquad
    ds\otimes d\mathbb Q\text{-a.e.}
\]
Together with $Y_T=u(T,X_T)=g(X_T)$, this shows that $(Y,Z)$ satisfies the optimal BSDE \eqref{equation:optimal-bsde}. As in the proof of
Theorem~\ref{theorem:joint-martingale-characterization}, \cite[Lemma~3.1]{BriandDelyonHuPardouxStoica2003} and uniqueness for \eqref{equation:optimal-bsde} imply that
\[
    Y_s=Y_s^{t,x,*},
    \quad
    s\in[t,T],
    \quad
    \mathbb Q\text{-a.s.},\qquad\qquad Z_s = Z^{t,x,*}_s,\quad ds\otimes d\mathbb Q\text{-a.e.} 
\]
In particular, $u(t,x)=u_*(t,x)$ for almost every initial condition. Hence $u=u_*$ $dt\otimes dx$-almost everywhere, and continuity on $[0,T)\times\mathbb R^d$, together with the terminal condition \eqref{equation:optimal-candidate-terminal}, yields
\[
    u=u_*
    \qquad\text{on }[0,T]\times\mathbb R^d.
\]
It remains to identify $q$. With $u=u_*$, the same
finite-variation and density argument as in
Theorem~\ref{theorem:q-martingale-characterization} yields
\[
    q=q_*
    \qquad
    dt\otimes dx\otimes da\text{-a.e.}
\]
This proves \textnormal{(i)}.

\medskip\noindent Finally, if $\exp(q/\gamma)$ admits an admissible representative,
then $q=q_*$ implies that this representative agrees almost
everywhere with \eqref{equation:optimal-policy}. The optimality
conclusion therefore follows from
Proposition~\ref{proposition:optimal-verification}.
\end{proof}

\paragraph{Acknowledgements.}
E.~Bandini, A.~Cosso, and L.~Perelli acknowledge support from GNAMPA-INdAM.

\small
\bibliographystyle{plainnat}
\bibliography{bibliography}

\end{document}